\documentclass[a4paper,11pt]{article}
\usepackage[utf8]{inputenc}
\usepackage[T1]{fontenc}
\usepackage[english]{babel}
\usepackage{a4wide}

\usepackage[]{hyperref}
\hypersetup{
    colorlinks=true,       
    linkcolor=red,          
    citecolor=blue,        
    filecolor=magenta,      
    urlcolor=cyan           
}

\usepackage[leqno]{amsmath}
\usepackage{amssymb}
\usepackage{mathrsfs}
\usepackage{amsthm}
\usepackage{amsxtra}
\usepackage{bm}
\usepackage[titletoc]{appendix}
\usepackage{mathrsfs}

\usepackage{graphicx}

\theoremstyle{plain}
\newtheorem{thm}{Theorem}[section]
\newtheorem{prop}[thm]{Proposition}
\newtheorem{lem}[thm]{Lemma}

\theoremstyle{definition}
\newtheorem{defn}[thm]{Definition}

\theoremstyle{remark}

\numberwithin{equation}{section}

\title{A priori estimates for water waves with emerging bottom}
\author{Thibault de Poyferr\'{e}\footnote{UMR 8536 CNRS, Centre de mathématiques et de leurs applications, \'Ecole Normale Supérieure Paris-Saclay, 94235 Cachan, France. Email: thibault.de.poyferre@cmla.ens-cachan.fr}}
\date{}

\DeclareMathOperator{\dive}{div}

\renewcommand{\c}[1]{\mathcal{#1}}
\renewcommand{\b}[1]{\mathbb{#1}}
\newcommand{\nab}[1]{\nabla_{\!\!#1}}
\def\d{\,\mathrm{d}}
\def\C{\bm{\mathrm{C}}}
\def\D{\bm{\mathrm{D}}}
\def\eps{\varepsilon}
\def\Id{\mathrm{Id}}
\def\la{\left\vert}
\def\lA{\left\Vert}
\def\lb{\left[}
\def\lB{\left\{}
\def\lp{\left(}
\def\ls{\left\langle}

\def\mez{\frac{1}{2}}
\def\N{\bm{\mathrm{N}}}

\def\R{\bm{\mathrm{R}}}
\def\ra{\right\vert}
\def\rA{\right\Vert}
\def\rb{\right]}
\def\rB{\right\}}
\def\rp{\right)}
\def\rs{\right\rangle}

\def\tr{\mathrm{tr}}

\def\Z{\bm{\mathrm{Z}}}
\renewcommand{\Re}{\mathrm{Re}}
\renewcommand{\Im}{\mathrm{Im}}

\begin{document}

\maketitle

\begin{abstract}
We study the beach problem for water waves. The case we consider is a compact fluid domain, where the free surface intersect the bottom along an edge, with a non-zero contact angle.  Using elliptic estimates in domain with edges and a new equation on the Taylor coefficient, we establish a priori estimates, for angles smaller than a dimensional constant. Local existence will be derived in a following paper.
\end{abstract}

\section{Introduction}
Suppose we are given a fixed smooth simply connected domain~$\mathcal{O}$ of~$\R^n$, with~$n\geq2$. We call~$\mathcal{M}$ its boundary, which we assume to be connected. An incompressible fluid fills a time-dependent domain~$\Omega_t\subset\mathcal{O}$, delimited by~$\mathcal{M}$ and a time-dependent hypersurface~$\mathcal{S}_t$. We assume this surface to be at all times connected, and such that~$\Omega_t$ is always compact and simply connected. The part of~$\mathcal{M}$ that bounds~$\Omega_t$, called the bottom, is thus time-dependent. We denote it by~$\mathcal{B}_t$.

Our last hypothesis on the domain is that for all times the intersection between~$\mathcal{M}$ and~$\mathcal{S}_t$ is along a time-dependent compact codimension~$2$ submanifold, the water line~$\mathcal{L}_t$. This intersection is assumed to be transverse, so that in particular the contact angle along~$\mathcal{L}_t$ is bounded away from~$0$ on each compact interval of time.

The motion of the fluid is described by its velocity~$v$ with values in~$\R^n$ defined for each~$t$ in the domain~$\Omega_t$, and satisfying the incompressible Euler equations in a constant gravity field,
\begin{equation*} \tag{E} \label{E}
	\lB
	\begin{aligned}
		&\partial_tv+v\cdot\nabla v=-\nabla p-ge_n,\\
		&\nabla\cdot v=0,
	\end{aligned}
	\right.
\end{equation*}
where for each time~$t$, the function~$p:\Omega_t\rightarrow\R$ is the pressure of the fluid. The constant~$g\geq0$ measures gravity, and~$e_n$ is a fixed unitary length vector which we think of as the upward direction.
The fluid domain moves with the velocity field, and pressure at the boundary is~$0$, so that
\begin{equation*} \tag{BC} \label{BC}
	\lB
	\begin{aligned}
		&\D_t:=\partial_t+v\cdot\nabla\text{ is tangent to }\cup_t\Omega_t\subset\R^{n+1},\\
		&p(t,x)=0, \quad x\in\mathcal{S}_t.
	\end{aligned}
	\right.
\end{equation*}
Here~$\D_t$ is the material derivative, and the first condition equivalently says that the velocity of~$\mathcal{S}_t$ is given by~$\ls v, N\rs$ with~$N$ the unit outward normal to~$\c S_t$, and that $\ls v,\nu\rs=0$ with~$\nu$ the unit outward normal on~$\c M$.

At a time~$t$ and a point~$x\in\c L_t$ of the water line, the angle between~$\c B_t$ and~$\c S_t$ measured in the plane of~$\nu$ and~$N$, will be written~$\omega(x)$.

Our objective is to give a local well-posedness result for the associated Cauchy problem. In this paper, we concern ourselves with establishing a priori estimates.
The study of the water waves problem has a long story, starting with Cauchy in~\cite{CauchyWW}. The rigorous derivation of local existence in Sobolev spaces, however, was only established in~1997, by Wu (\cite{WuWPos2D,WuWPos3D}). Then a number of other proofs, improving on the regularity, the shape of the bottom, or using different approaches, appeared in the last~20 years. We only quote Beyer and Gunther in~\cite{BeyerGuntherCapDrop}, Christodoulou and Lindblad in~\cite{ChristodoulouLindbladMotFrSur},  Lannes in~\cite{LannesWellPos}, Coutand and Shkoller in~\cite{CoutandShkollerWW}, Alazard, Burq and Zuily in~\cite{AlazardBurqZuilyExwST,AlazardBurqZuilyExw/ST}, Hunter, Ifrim and Tataru in~\cite{HunterIfrimTataru2DWW}, and with vorticity, Castro and Lannes in~\cite{CastroLannesVortForm}, and Shatah and Zeng in~\cite{ShatahZeng1,ShatahZeng2,ShatahZeng3}.

However, all of those papers assume a laterally infinite ocean, where there is no contact line. One trick to study such a configuration, when the walls are vertical, is to periodize and symmetrize. This was done by Alazard, Burq and Zuily in~\cite{AlazardBurqZuilyExULoc}, for right angles, and later developed in the case of general angles by Kinsey and Wu (\cite{KinsleyWuAngled}) and then by Wu (\cite{Wuangled}). The case of a more general angle has only, to the best of our knowledge, been tackled by Ming and Wang (\cite{MingWangDNCorner}). In their paper, they study the Dirichlet to Neumann operator associated with such a configuration, in 2D, and give a complete description of its singularities at the corner.

The model of the Euler equation, with the boundary condition described above, is only an approximation of the real physical phenomenon. In practice, viscosity and surface tension are fundamental to a precise description of the motion close to the corner. Steps in this direction have been done by Guo and Tice for stability of the equilibrium in~\cite{GuoTiceStabContLin}, and by Tice and Zheng in~\cite{TiceZhengLocWPCont} for well-posedness, both for the Stokes flow.

Our theorem, stated informally, is as follows. The notation~$H^s$ is for the Sobolev spaces based on~$L^2$.
\begin{thm}
 Suppose~$\c S_t$, a~$C^2$ in time family of~$H^s$ hypersurfaces, and~$v\in C^2(H^s(\Omega_t))$, are solution of the equations.
 
 Here~$s>1+\frac n2$, and~$s<\frac12+\frac\pi{2\overline\omega}$, where~$\overline\omega>0$ is a number, such that for all~$t$, for all~$x$ in~$\c L_t$, $\omega(x)\leq\overline\omega$. Notice this implies~$\omega<\frac\pi{n+1}$.
 
 Assume also that there is a number~$a_0>0$ such that the Taylor coefficient~$a:=-\nab Np\geq a_0>0$ for all~$t$, and a number~$\underline \omega>0$ such  that~$\omega\geq\underline\omega$ for all~$t$.
 
 Then, for some energy~$E(t)=E\lp\Omega_t,v(t,\cdot)\rp$, to be defined below, and controlling~$\c S_t$ in~$H^s$ and~$v\in H^s(\Omega_t)$, there exists a time~$T>0$, depending only on the norms of the initial data, such that for all~$t$ in~$[0,T]$, 
 \[E(t)\leq E(0)+\int_0^t\c F\lp E(t')\rp\d t',\]
 where~$\c F$ is an increasing function depending only on~$\underline\omega$, $s$, $a_0$, and a neighborhood of the initial data in the rougher topology~$H^{s-\frac12}\times H^{s-\frac12}(\Omega_t)$.
\end{thm}

To state this Theorem precisely, we need to prescribe the topology on~$H^s$ hypersurfaces, which will be done in~\ref{sec:analysis}, and to define the Energy. Then Proposition~\ref{prop:controlE} gives the control of the unknowns from~$E$ and Proposition~\ref{prop:Gron} gives the time~$T$ and the estimation on the evolution of~$E$.

A few remarks are in order. First, in the classical case of a well-separated bottom and free surface, we would have the same Theorem, without the restrictions on the angle. The level of regularity, which corresponds to~$v\in C^1$ by Sobolev embedding, is the best we can find without using dispersive properties (see for example~\cite{AlazardBurqZuilyExw/ST}). Notice we do not assume the field~$v$ to be irrotational.

Second, the condition on the angle arise because of the presence of an edge in the domain. In such a domain, the elliptic regularity theory degenerates. This elliptic regularity is needed to make sense of the equations, since~$p$ solves an elliptic problem. It also comes into play often in the analysis. The allowed range for~$s$ is the one where elliptic regularity works as in smooth domains, as will be seen in Section~\ref{sec:analysis}.

Last, we expect to be able to prove local well-posedness for the same problem, under the same regularity and with the conditions on the Taylor coefficients and the angle being satisfied initially, for a time depending only on the norms of the initial data~$a_0$, and~$\underline\omega$. 

In Section~\ref{sec:Geom}, we study heuristically the infinite-dimensional geometry of the problem, derive the linearized equation, and explain its consequences on our strategy. In Section~\ref{sec:analysis}, we develop all the analytical tools needed to study moving hypersurfaces with boundaries and moving domains with edges, in particular the elliptic regularity theory. Since the problem is fully non-linear, a classical strategy to prove a priori estimates is to differentiate the equation to reduce it to a quasilinear form, which we hope to be equivalent to the original one. Usually, one would differentiate in space. However, this only work for translation-invariant equations, which is not the case of this problem. Instead, we take advantage of the time-translation invariance, and differentiate in time. This is accomplished in Section~\ref{sec:quasilin}. At last, the Energy is defined and studied in Section~\ref{sec:Energy}, where the two main Propositions are stated. 

In our analysis, we are heavily indebted to two works, from which we draw heavily. The first is the book by Dauge, \cite{DaugeEllCor}, from which we take the analysis of the elliptic problem. Our main contribution in this domain is to precise the dependence of the constants in the geometry. To the best of our knowledge, the div-curl lemma is new.
The second work is the series of three articles by Shatah and Zeng, \cite{ShatahZeng1,ShatahZeng2,ShatahZeng3}, who developed a coordinate-free framework for the analysis of the water waves problem. Although the analysis ends up being quite different, due to the failure of the mean curvature to quasi-linearize the equations, the coordinate-free framework, most of the notations, and a few computations come directly from there.

\section*{Acknowledgment}
 The author would like to sincerely thank his advisor, T. Alazard, for his support and comments, and B. Pausader for a very helpful discussion.

\section{Geometry of the problem}	\label{sec:Geom}
In this section we heuristically derive the linearized equation from the energy. In particular, we do not discuss the smoothness of the objects involved. Most of this section is from the work~\cite{ShatahZeng1} by Shatah and Zeng, where they study the case of a droplet. We show that this heuristic analysis stays valid in our case, and explain its consequences for our strategy.
\subsection{Lagrangian formulation}
Under the conditions~\eqref{BC}, the Euler equation~\eqref{E} is easily seen to admit a conserved energy
\[E_0=\int_{\Omega_t}\frac{\la v\ra^2}{2}\d x+g\int_{\Omega_t}x^n\d x,\]
where~$x_n$ is the coordinate of~$x$ along~$e^n$. We want to express~\eqref{E} as a minimizer of an associated Lagrangian, under the constraints~\eqref{BC}.

For this, we introduce the Lagrangian coordinates by solving the ODE
\[\frac{\d x}{\d t}=v(t,x),\quad x(0)=y,\]
which gives the spatial path of a fluid particle initially at position~$y\in\Omega_0$. Then we introduce for each~$t$ the diffeomorphism~$u(t,\cdot):\Omega_0\rightarrow\Omega_t$ as the flow of this ODE.
The divergence free condition on~$v$ induces that~$u$ preserves the volume, and now~$v=u_t\circ u^{-1}$.
For any vector field~$w$ on~$\Omega_t$, we write~$\bar{w}=w\circ u$ defined on~$\Omega_0$, and the chain rule implies
\begin{equation}	\label{eq:Eultder}
	\D_tw=\partial_tw+\nab vw=\bar{w}_t\circ u^{-1}.
\end{equation}
Here and in all the following, $\nab v w:=\ls v,\nabla w\rs$ where~$\ls,\rs$ is the scalar product.

A solution of the Euler equation is thus a path, starting from identity, in the infinite dimensional manifold
\[\Gamma:=\lB \Phi:\Omega_0\rightarrow\R^n\;|\;\Phi\text{ is volume preserving and }\Phi(\c B_0)\subset\c M\rB.\]
Its tangent space at the point~$\Phi$ is
\[T_\Phi\Gamma:=\lB\bar w:\Omega_0\rightarrow\R^n\;|\;\nabla\cdot w=0\text{ on }\Phi(\Omega_0)\text{ and } w\cdot\nu=0\text{ on }\Phi(\c B_0), \text{ for }w=\bar w\circ\Phi^{-1}\rB.\]
The energy now takes the form 
\[E_0=\frac12\la u_t\ra^2_{L^2(\Omega_0)}+g\b G(u):=\int_{\Omega_0}\frac{\la u_t\ra^2}2\d y+g\int_{\Omega_0}u^n\d y.\]
This suggest that the associated Lagrangian action is
\[\int\b L(u)\d t=\iint_{\Omega_0}\frac{\la u_t\ra^2}2\d y\d t-g\int\b G(u)\d t.\]

It then follows from classical variational principles that a minimizer of this action is a path~$u$ in~$\Gamma$ whose velocity~$v(t)$ should satisfy the equation
\begin{equation}	\label{eq:EulLag}
	\bar{\mathscr{D}}_tu_t+g\b G'(u)=0.
\end{equation}
Here~$\bar{\mathscr{D}}$ is the covariant derivative on~$\Gamma$ for the~$L^2$ metric. We notice that~$\Gamma$ is a submanifold of the space of diffeomorphisms from~$\Omega_0$, equipped with~$L^2$ metric, whose tangent space is simply the space of vector fields~$\bar w$ on~$\Omega_0$. Its covariant derivative along a path~$u$ is simply~$\bar w_t$, so that
we have for an element~$\bar w\in T\Gamma$ defined above the path~$u(t)$ the formula
\begin{equation}	\label{eq:cov1}
	\bar{\mathscr{D}}\bar w=\bar w_t-\mathrm{II}_u(u_t,\bar w).
\end{equation}
Here~$\mathrm{II}$ is the second fundamental form of~$\Gamma$ as a submanifold of this space of diffeomorphisms. 
\paragraph{Hodge decomposition.}
To compute~$\mathrm{II}(u_t,\bar w)$ we observe that any vector field~$X$ in~$\Omega$ can be decomposed uniquely as
\begin{equation*}
	X=w+\nabla\phi,
\end{equation*}
where~$\phi$ is defined as the solution of 
\begin{equation} \label{eq:Hodgeperp}
	\lB
	\begin{aligned}
		&\Delta\phi=\nabla\cdot X\text{ in }\Omega,\\
		&\phi\rvert_{\c S}=0,\\
		&\nab \nu\phi\rvert_{\c B}=\ls X,\nu\rs.
	\end{aligned}
	\right.
\end{equation}
 Thus~$w$ verifies
\begin{equation*} 
	\lB
	\begin{aligned}
		&\nabla\cdot w=0\text{ in }\Omega,\\
		&w\rvert_{\c S}=X\rvert_{\c S},\\
		&\ls w,\nu\rs\rvert_{\c B}=0.
	\end{aligned}
	\right.
\end{equation*}
This decomposition is easily seen to be~$L^2$ orthogonal. Thus we can identify
\begin{equation}	\label{eq:perpGamma}
	(T_\Phi\Gamma)^\perp=\lB-(\nabla\phi)\circ\Phi\;|\;\phi\rvert_{\c S}=0\rB.
\end{equation}
Keep in mind however that, since in~\eqref{eq:Hodgeperp} we define~$\phi$ by inverting the Laplace operator in a domain with corner, the parts~$w$ and~$\nabla\phi$ of the decomposition are not necessarily smooth, even if~$X$ is. 
\paragraph{Covariant derivative.}
Now coming back to~\eqref{eq:cov1}, we see that for a path~$u(t)$ in~$\Gamma$ with velocity field~$u_t=\bar v$, and a vector field~$\bar w$ defined on it, there holds
\begin{equation*}
	\bar{\mathscr{D}}_t\bar w=\bar w_t+(\nabla p_{v,w})\circ u,
\end{equation*}
where
\begin{equation*}
	\lB
	\begin{aligned}
		&\Delta p_{v,w}=-\mathrm{tr}(DvDw)\text{ in }\Omega,\\
		&p_{v,w}\rvert_{\c S}=0,\\
		&\nab \nu p_{v,w}\rvert_{\c B}=-\Pi_{\c M}(v, w),
	\end{aligned}
	\right.
\end{equation*}
with~$\Pi_{\c M}$ the second fundamental form of the bottom. This can be inferred from~\eqref{eq:Hodgeperp} by taking~$X=\bar{w}_t\circ u^{-1}=\D_tw$.

Now this is in Lagrangian coordinates, and it can be rewritten in Eulerian coordinates, using~\eqref{eq:Eultder}. This gives
\begin{equation*}
	\bar{\mathscr{D}}_t\bar w=\lp\D_tw+\nabla p_{v,w}\rp\circ u=\lp\partial_tw+\nab vw+\nabla p_{v,w}\rp\circ u.
\end{equation*}
\paragraph{Gravity force.}
We then compute~$\b G'(u)$. For any~$\bar w\in T_u\Gamma$, take a path in~$\Gamma$ indexed by~$\eps$ and starting from~$u$ at~$\eps=0$, with tangent vector at~$\eps=0$ equal to~$\bar w$. Then
\begin{align*}
	\ls\b G'(u),\bar w\rs_{L^2(\Omega_0)}&=\frac{\d}{\d\eps}\int_{u(\Omega_0)}x^n\d x\\
					     &=\int_{u(\Omega_0)}\D_\eps x^n\d x\\
					     &=\int_{u(\Omega_0)}\ls w,\nabla x^n\rs\d x\\
					     &=\int_{u(\c S_0)}w^\perp x^n\d S\\
					     &=\int_{u(\Omega_0)}\ls w,\nabla\c H\lp x^n\rvert_{u(\c S_0)}\rp\rs\d x.
\end{align*}
Here we have used the Green formula twice, and the terms on~$u(\c B)$ vanish since~$\ls w,\nu\rs=0$ there. We have replaced~$\nabla x^n$ with~$\nabla\c H\lp x^n\rvert_{u(\c S_0)}\rp$, where~$\c H$ is the harmonic extension with homogeneous Neumann condition on the bottom, so that now~$\nabla\c H\lp x^n\rvert_{u(\c S_0)}\rp\in T_u\Gamma$ and we can identify it with~$\b G'(u)$.

Then the Euler-Lagrange equation~\eqref{eq:EulLag} for our action become in Eulerian coordinates
\begin{equation}
	\partial_tv+v\cdot\nabla v=-\nabla p_{v,v}-g\nabla\c H\lp x^n\rvert_{\c S_t}\rp=-\nabla p-g e^n,
\end{equation}
with~$p=p_{v,v}+g (e^n-\c H\lp x^n\rvert_{\c S_t}\rp)$ the physical pressure. Combined with the constraint that~$v\circ u\in T_u\Gamma$ is the velocity vector of the domain, this gives the Euler equations~\eqref{E} with boundary conditions~\eqref{BC}.
\subsection{The linearized equation}
To help us study the Euler equations, we want to find a way to linearize them around a given solution, i.e. a path~$u(t)$ in~$\Gamma$, such that its velocity~$\bar v=u_t\in T_u\Gamma$ satisfies the Euler equations.
Since we see the Euler equation as a geodesic  flow with potential, the natural linearization is through the Jacobi equation. It is the equation that a time-dependent vector field~$\bar w\in T_u(t)\Gamma$ defined on the path~$u$ has to satisfy if, by moving the curve~$u$ by the flow of~$\bar w$, we want it to stay a solution of the Euler equations. Classically, this is 
\begin{equation}
	\bar{\mathscr{D}}_t^2\bar w+\bar{\mathscr{R}}(u_t,\bar w)u_t+g\bar{\mathscr{D}}^2\b G(u)\bar w=0,
\end{equation}
where~$\bar{\mathscr{R}}$ is the Riemann curvature tensor of~$\Gamma$ at the point~$u(t)$, and~$\bar{\mathscr{D}}^2\b G(u)$ is the Hessian of~$\b G$. 
We need to compute~$\bar{\mathscr{R}}(u_t,\bar w)u_t$ and~$\bar{\mathscr{D}}^2\b G(u)\bar w$, or at least their principal parts,  from their bilinear forms. To do this, we consider for a given~$\bar w\in T_u(t)\Gamma$ a family of curves~$u(t,\eps)\in\Gamma$ such that~$u(t,0)=u(t)$, and~$\partial_\eps u(t,0)=\bar w$. Then we extend~$\bar w$ to be the tangent vector in~$\eps$. 
\paragraph{The Riemann curvature}
We use the classical formula
\begin{align*}
	\ls\bar{\mathscr{R}}(\bar v,\bar w)\bar v,\bar w\rs_{L^2(\Omega_0)}&=\ls\mathrm{II}_u(\bar v,\bar v),\mathrm{II}_u(\bar w,\bar w)\rs_{L^2(\Omega_0)}-\la\mathrm{II}_u(\bar v,\bar w)\ra^2_{L^2(\Omega_0)}\\
								    &=\int_{\Omega_t}\ls\nabla p_{v,v},\nabla p_{w,w}\rs\d x-\int_{\Omega_t}\la\nabla p_{v,w}\ra^2\d x.
\end{align*}

Then assuming enough regularity on~$v$ and~$w$, we can repetitively use the Green formula and the definition of~$p_{.,.}$to find
\begin{align*}
	\int_{\Omega_t}\ls\nabla p_{v,v},\nabla p_{w,w}\rs\d x&=\int_{\Omega_t}p_{v,v}\mathrm{tr}(DwDw)\d x+\int_{\c B_t}p_{v,v}\nab \nu p_{w,w}\d S\\
							    &=-\int_{\Omega_t}\ls\nabla p_{v,v},\nab ww\rs\d x+\int_{\c B_t}p_{v,v}(\ls\nab w\nu,w\rs+\ls\nab ww,\nu\rs)\d S\\
							    &=\int_{\Omega_t}D^2p_{v,v}(w,w)\d x-\int_{\c S_t}w^\perp\nab wp_{v,v}\d S.
\end{align*}
Here we have also used the identity 
\[\ls\nab w\nu,w\rs+\ls\nab ww,\nu\rs=w\ls w,\nu\rs=0,\]
where~$w$ is taken as a derivation, because~$\ls w,\nu\rs=0$ on~$\c B$, and~$\nabla\cdot w=0$.

Then~$p_{v,v}=0$ on~$\c S_t$, and thus~$\nab wp_{v,v}=w^\perp\nab Np_{v,v}$. A last application of the Green formula gives
\[\int_{\Omega_t}\ls\nabla p_{v,v},\nabla p_{w,w}\rs\d x=\int_{\Omega_t}\ls w,\nabla\c H(-\nab Np_{v,v}w^\perp)\rs\d x+\int_{\Omega_t}D^2p_{v,v}(w,w)\d x.\]

Now the second term is expected to be more regular, so that
\begin{equation*}
	\bar{\mathscr{R}}(\bar v,\bar w)\bar v\eqsim(\mathscr{R}_0(v)w)\circ u 
\end{equation*}
where~$\mathscr{R}_0(v)w=\nabla\c H(-\nab Np_{v,v}w^\perp)$.

\paragraph{The gravity term}
To compute~$\b G''(u)$, we use the formulas~\eqref{eq:EvN} and~\eqref{eq:EvSur} for the evolutions of the normal and the surface element of a surface moving with divergence-free velocity~$w$:
\begin{align*}
	\ls\b G''(u)\bar w,\bar w\rs_{L^2(\Omega_0)}&=\partial_\eps\ls\b G'(u),\bar w\rs_{L^2(\Omega_0)}-\ls\b G'(u),\bar{\mathscr{D}}_\eps\bar w\rs_{L^2(\Omega_0)}\\
						 &=\partial_\eps\int_{\c S}x^nw^\perp\d S-\int_{\c S}x^nN\cdot(\D_\eps w+\nabla p_{w,w})\d S\\
						 &=\int_{\c S} w^nw^\perp-x^nw\cdot\lp(Dw)_*(N)\rp^\top+x^nN\cdot\D_\eps w\\
						 &\qquad-x^nw^\perp\nab Nw\cdot N-x^nN\cdot\D_\eps w-x^nN\cdot\nabla p_{w,w}\d S\\
						 &=\int_{\c S}\lp w^nw^\perp-x^n\nab ww\cdot N-x^nN\cdot\nabla p_{w,w}\rp\d S.
\end{align*}
But using repeated Green formulas give
\begin{align*}
	-\int_{\c S}x^nN\cdot\nabla p_{w,w}\d S&=-\int_\Omega\nabla\c H(x^n\rvert_{\c S})\cdot\nabla p_{w,w}+\c H(x^n\rvert_{\c S})\mathrm{tr}\lp(Dw)^2\rp\d x\\
	&\qquad+\int_{\c B}\c H(x^n\rvert_{\c S})\nu\cdot\nabla p_{w,w}\d S\\
					      &=-\int_\Omega\nabla\c H(x^n\rvert_{\c S})\cdot\nab ww\d x+\int_{\c S}x^n\nab ww\cdot N\d S\\
					      &\qquad+\int_{\c B}\c H(x^n\rvert_{\c S})(\nab ww\cdot\nu+\nab w\nu\cdot w)\d S\\
					      &=\int_\Omega D^2\c H(x^n\rvert_{\c S})(w,w)\d x+\int_{\c S}x^n\nab ww\cdot N-w^\perp\nab w\c H(x^n\rvert_{\c S})\d S.
\end{align*}
Noticing that
\[\nab w\c H(x^n\rvert_{\c S})=\nab{w^\top}x^n+w^\perp\c N(x^n\rvert_{\c S})=w^n-w^\perp N^n+w^\perp\c N(x^n\rvert_{\c S}),\]
we find
\begin{align*}
	\ls\b G''(u)\bar w,\bar w\rs_{L^2(\Omega_0)}&=\int_{\c S}\lp w^\perp\rp^2\lp N^n-\c N(x^n\rvert_{\c S})\rp\d S+\int_\Omega D^2\c H(x^n\rvert_{\c S})(w,w)\d x\\
						    &=\int_\Omega w\cdot\nabla\c H\lb\lp N^n-\c N(x^n\rvert_{\c S})\rp w^\perp\rb\d x+\int_\Omega D^2\c H(x^n\rvert_{\c S})(w,w)\d x.
\end{align*}

Again the second term is more regular, so that
\begin{equation*}
	\b G''(u)\bar w\simeq(\mathscr Gw)\circ u 
\end{equation*}

where~$\mathscr Gw=\nabla\c H\lb\lp N^n-\c N(x^n\rvert_{\c S})\rp w^\perp\rb$.

Thus the linearized equation becomes in Eulerian coordinates
\begin{equation}	\label{eq:lin}
	\mathscr{D}_t^2w+\mathscr{R}_0(v)w+g\mathscr Gw=\text{bounded terms}.
\end{equation}
we observe that both~$\mathscr R_0$ and~$\mathscr G$ are of order~$1$, and that their forms are similar. In fact, we can write
\begin{equation}	\label{eq:defRg}
	(\mathscr R_0+g\mathscr G)w=\nabla\c H(aw^\perp)=:\mathscr R_gw,
\end{equation}
where~$a$ is the Rayleigh-Taylor coefficient
\begin{equation}	\label{eq:RayTayCoef}
	a:=-\nab N\lp p_{v,v}+\c H(x^n\rvert_{\c S})-x^n\rp=-\nab Np,
\end{equation}
where~$p$ is again the physical pressure.

\paragraph{The Rayleigh-Taylor coefficient}
It can be seen on this equation that there is stability (meaning exponential decay of the solution) only if the Rayleigh-Taylor condition
\begin{equation}
	a(t,x)\geq c>0,\forall x\in\Omega_t
\end{equation}
holds for all times, with~$c$ an arbitrary positive constant.

Assuming enough regularity, we can compute~$a$ at the triple line.
There holds
\begin{equation*}
	\nab\nu p_{v,v}=\ls\nu,\nabla p_{v,v}\rs=\ls\nu^\top,\nabla^\top p_{v,v}\rs+\ls\nu, N\rs\nab Np_{v,v}=\ls\nu, N\rs\nab Np_{v,v},
\end{equation*}
because~$p_{v,v}=0$ on~$\c S$. Here, $A^\top$ refer to the tangent part at~$\c S$. On the other hand,
\begin{equation*}
	\nab\nu p_{v,v}=\ls v,\nab v\nu\rs=-\ls\nu,\nab vv\rs,
\end{equation*}
because~$\ls v,\nu\rs=0$ on~$\c B$. Therefore, assuming~$\ls\nu,N\rs\neq0$, we find
\begin{equation}	\label{eq:hydroTay}
	-\nab Np_{v,v}=\frac{\nu\cdot\nab vv}{\nu\cdot N}.
\end{equation}

A similar computation can be performed on the gravity part:
\begin{multline*}
	0=\nab\nu\c H(x^n\rvert_{\c S})=\ls\nu^\top,\nabla^\top\c H(x^n\rvert_{\c S})\rs+\ls\nu,N\rs\nab N\c H(x^n\rvert_{\c S})\\
	=\ls e^n,\nu-\ls\nu,N\rs N\rs+\ls\nu,N\rs\nab N\c H(x^n\rvert_{\c S}),
\end{multline*}
and since~$\nab N x^n=N^n$, we find
\begin{equation}	\label{eq:graTay}
	-g\nab N(\c H(x^n\rvert_{\c S})-x^n)=g\frac{\nu^n}{\ls\nu,N\rs},
\end{equation}
again assuming~$\ls\nu,N\rs\neq 0$.

Thus putting together~\eqref{eq:hydroTay} and~\eqref{eq:graTay} gives for~$\ls\nu,N\rs\neq 0$ that
\begin{equation}
	a=\frac{g\nu^n+\ls\nu,\nab vv\rs}{\ls\nu,N\rs}.
\end{equation}

To see what this means, we specialize to 2D situations, with zero velocity field. Then~$\ls\nu,N\rs=-\cos(\omega)$, with~$\omega$ the angle between the bottom and the free surface, so that the condition~$a>0$ gives the situation of an acute angle and where the water is above the bottom, which would be the case of a beach, as stable.

Of course, when the velocity field is non zero, it can counterbalance the effect of gravity and change this situation.

\section{Analysis on moving domains}	\label{sec:analysis}
In this section, we develop the norms and estimates we need for our analysis.
The main objective is to derive estimates whose constants are independent of the domain.

\subsection{Surface coordinates}
Our first objective is to give a description in coordinates of the hypersurfaces in a given~$H^{s_0}$ neighborhood. Take~$s_0>\frac{n+1}2$.
Using local coordinates, one can easily define what it means to be an~$H^r$ function on a given~$H^{s_0}$ hypersurface with boundary~$\c S$. For~$s_0>r>-s_0$, those are simply functions whose coordinates representatives are locally in~$H^r(\R^{n-1})$ for interior coordinates and~$H^r(\R^{n-1}_+)$ for boundary coordinates. Here~$\R^{n-1}_+$ is the upper half-plane, and~$H^r$ functions are simply restrictions of functions that are~$H^r$ in the whole plane. 

It is easy to see that this produce a Banach space, and a norm can be chosen by taking a covering of~$\c S$ by a finite number of coordinates patch, and an adapted partition of unity. However such a norm is dependent on those choices of coordinates, and therefore we will not use it. 
After that, one can define a topology on the space of~$H^{s_0}$ surfaces with boundary contained in our fixed bottom hypersurface~$\c M$, by saying that two are close if a diffeomorphism from one to the other is close to identity in~$H^{s_0}$ norm. It is quite easy to see that the subspace of such surfaces whose intersection with~$\c M$ is transverse is an open set, and therefore we can consider a neighborhood of a given smooth hypersurface~$\c S_*$ consisting entirely of~$H^{s_0}$ surfaces intersecting~$\c M$ transversally. By density, any hypersurface is included in one such neighborhood.

Now we will construct such a neighborhood. Take some compact, smooth, reference hypersurface~$\c S_*$,  whose intersection with~$\c M$ is transverse, and whose boundary is this intersection~$\c L_*$. We want to represent close enough surfaces as graphs over~$\c S_*$, and for this we need a good collar neighborhood of~$\c S_*$. We cannot use normal coordinates because since we want to represent surfaces with boundary contained in~$\c M$, we need to straighten~$\c M$ in some way. We accomplish this through the following lemma. Recall that~$\c O$ is the domain whose boundary is~$\c M$.
\begin{lem}	\label{lem:vectcoord}
	There exists a smooth unit vector field~$X$, defined on a neighborhood of~$\c S_*$ in~$\c O$, that is always transverse to~$\c S_*$ and always tangent to~$\c M$. 
	There exists~$\delta>0$ such that the flow of~$X$, 
	\[\phi:\c S_*\times[-\delta,\delta]\rightarrow\R^n,\]
	is a smooth diffeomorphism from its domain to a neighborhood of~$\c S_*$ in~$\c O$.
\end{lem}
\begin{proof}
	One start by constructing~$X_1$, always normal to~$\c S_*$ away from its boundary. For this, we consider~$\c S_*$ only as an hypersurface with boundary of~$\R^n$. We take an open submanifold of it, which is an  hypersurface of~$\R^n$. Now we take the normal to this hypersurface, and we extend it to a neighborhood in~$\R^n$.
	
	Then in a neighborhood of~$\c L_*$ in~$\c M$, we can construct a smooth vector field~$X_2$, tangent to~$\c M$ and transverse to~$\c L_*$, by extending the normal to~$\c L_*$ in~$\c M$. We can extend it in a neighborhood of~$\c L_*$ in~$\bar{\c O}$ to a smooth vector field tangent to~$\c M$ and transverse to~$\c S_*$, since their intersection is transverse.
	
	To finish, we can cover a small enough neighborhood of~$\c S_*$ in~$\bar{\c O}$ with this neighborhood where~$X_2$ is defined, and an open set whose closure is in the interior of~$\c O$, and where~$X_1$ is well defined. At last, we can use a partition of unity to glue them smoothly to form the vector field~$X$.
	
	The existence of~$\phi$, its regularity, and the fact that it is a diffeomorphism for small enough~$\delta$ is a simple consequence of the theory of ODEs.
\end{proof}
If we fix an~$H^{s_0}$ norm on~$\c S_*$, we can express a neighborhood of it in the space of~$H^{s_0}$ surfaces with boundary in~$\c M$ by the condition that there exits a diffeomorphism~$F$ between the two satisfying
\[\lA F-\Id\rA_{H^{s_0}(\c S_*;\R^n)}<\delta_1.\]
For~$\delta_1$ small enough, this implies transversality of all the surfaces in the neighborhood. Taking again~$\delta_1$ small enough, those surfaces are all contained in the collar neighborhood we just defined. Then in those collar coordinates, again for~$\delta_1$ small, those are necessarily graphs above~$\c S_*$. Therefore, we can represent such a neighborhood by functions~$\eta_{\c S}$ defined on~$\c S_*$ with small enough~$H^{s_0}$ norms, and those give diffeomorphisms 
\begin{equation*}
	\Phi_{\c S}(p):=\phi(p,\eta_{\c S}(p))
\end{equation*}
in~$H^s(\c S_*;\R^n)$. 

All those notions can be restricted to~$\c L_*$, so that~$\c L$ is the graph of a function~$\eta_{\c L}$ which is the trace of~$\eta_{\c S}$ on~$\c L_*$, giving a diffeomorphism~$\Phi_{\c L}$ which is also the trace of~$\Phi_{\c S}$. If~$n=2$, $\c L$ is just two points, and these terms are well-defined because~$\eta_{\c S}$ is~$H^{s_0}$, and therefore continuous. For~$n\geq3$, those are traces in the Sobolev sense, and those traces are well-defined in~$H^{s_0-\frac12}(\c L_*)$ since~$s_0>\frac{n+1}2$.

\begin{defn}
	For~$\delta>0$ and~$s_0>\frac{n+1}2$, we define~$\Lambda(\c S_*,s_0,\delta)$ as the neighborhood of~$H^{s_0}$ hypersurfaces~$\c S$ such that their associated~$\eta_{\c S}$ satisfies~$\lA\eta_{\c S}\rA_{H^{s_0}(\c S_*)}<\delta$.
\end{defn}

For surfaces~$\c S$ in~$\Lambda(\c S_*,s_0,\delta)$, we can define the Sobolev norms~$H^r(\c S)$, for~$-s_0\leq r\leq s_0$, from the reference norm on~$\c S_*$.
In the analysis of the evolution problem, we will work with surfaces in~$\Lambda_*:=\Lambda(\c S_*,s-\frac12,\delta)$, with~$s>1+n/2$, and where~$\delta>0$ is small enough that all the above properties hold. However, our surfaces will really be of~$H^s$ class. The reason for this is that we do not want any smallness condition in the norm where the dynamics take place. The set~$\Lambda_*$ takes the role of a control neighborhood, and by choosing~$\c S_*$ close to the initial surface~$\c S_0$, we can treat any case. 

Since being in~$\Lambda_*$ is sufficient to have a well-defined~$\Phi_{\c S}$, we can use its~$H^s$ norm to measure the regularity of~$\c S$. More precisely, for~$\c S\in\Lambda_*$, if both~$\c S_*$ and~$\c S$ are in~$H^s$, we write
\begin{equation}
 \la\c S\ra_{s}:=\lA \Phi_{\c S}\rA_{H^s(\c S_*)}.
\end{equation}
Of course, any other choice of reference surface~$\c S_*$ provides an equivalent quantity, as soon as both are defined.

We also write
\begin{equation}
 \la\c L\ra_{s-\frac12}:=\lA \Phi_{\c L}\rA_{H^{s-\frac12}(\c L_*)}
\end{equation}
in dimension~$n\geq3$. In dimension~$n=2$, $\c L$ consists only of two point, whose positions on~$\c M$ are controlled by the condition~$\c S\in\Lambda_*$, so that we do not need to control any regularity.

The procedure to prove estimates with constants uniform in~$\Lambda_*$ is to prove them on~$\c S_*$  and then pushing them forward to~$\c S$. If we do not study norms greater than~$H^{s-\frac12}(\c S)$, this will only involve the~$H^{s-\frac12}(\c S_*)$ norms of~$\Phi_{\c S}$ and~$\Phi_{\c S}^{-1}$, which are uniformly bounded for~$\c S\in\Lambda_*$. For example, it is easy to prove the following product estimates, which will be used freely in the paper.
\begin{prop}
 If~$s>1+\frac n2$, $\c S_*$ is a reference hypersurface, $\delta$ small enough and~$\c S\in\Lambda_*$, then for any functions~$f\in H^{s_1}(\c S)$ and~$g\in H^{s_2}(\c S)$, with~$s_1\leq s_2\leq s-\frac12$, there holds
 \begin{align*}
  \lA fg\rA_{H^{s_1+s_2-\frac{n-1}2}(\c S)}\leq C\lA f\rA_{H^{s_1}(\c S)}\lA g\rA_{H^{s_2}(\c S)}&\quad\text{if }s_2<\frac{n-1}2\text{ and }0<s_1+s_2,\\
  \lA fg\rA_{H^{s_1}(\c S)}\leq C\lA f\rA_{H^{s_1}(\c S)}\lA g\rA_{H^{s_2}(\c S)}&\quad\text{if }s_2>\frac{n-1}2\text{ and }0\leq s_1+s_2.
 \end{align*}
 Here~$C$ depends only on~$\Lambda_*$.
 
 Similar inequalities hold on~$\c L$ in dimension~$n\geq3$.

\end{prop}

\subsection{From the curvature to the surface}
Recall that the mean curvature~$\kappa$ of~$\c S$ is defined as the trace of the second fundamental form~$\Pi$, whose definition is in turn
\[\Pi(v,w)=-\ls\nab vN,w\rs\]
for~$v,w\in T\c S$. 

The regularity of the hypersurface~$\c S$ can be measured by its curvature~$\kappa$ and the curvature~$\kappa_l$ of its boundary~$\c L$ taken as a hypersurface of~$\c M$ (this is only needed in dimension greater than~$3$). 
More precisely, we have the following lemmas, distinguishing between dimension~$2$ and dimension greater than~$3$.
\begin{lem}
  For~$n=2$, take~$s>2$, a reference hypersurface~$\c S_*$, and~$\delta>0$ small enough. Assume the hypersurface~$\c S$ is in~$\Lambda_*$, and~$\kappa\in H^{s-2}(\c S)$. Then the surface~$\c S$ is actually~$H^s$, and we have the following estimates on its geometry:
 \[\la\c S\ra_s+\lA \Pi\rA_{H^{s-2}(\c S)}+\lA \N\rA_{H^{s-1}(\c S)}\leq C\lp1+\lA\kappa\rA_{H^{s-2}(\c S)}\rp.\]
\end{lem}
 
 \begin{lem}
  For~$n\geq3$, take~$s>1+\frac n2$, a reference hypersurface~$\c S_*$, and~$\delta>0$ small enough. Assume the hypersurface~$\c S$ is in~$\Lambda_*$, and~$\kappa\in H^{s-2}(\c S)$, $\kappa_l\in H^{s-\frac52}(\c L)$. Then the surface~$\c S$ is actually~$H^s$, and we have the following estimates on its geometry:
 \[\la\c S\ra_s+\lA \Pi\rA_{H^{s-2}(\c S)}+\lA \N\rA_{H^{s-1}(\c S)}\leq C\lp1+\lA\kappa\rA_{H^{s-2}(\c S)}+\lA\kappa_l\rA_{H^{s-\frac52}(\c L)}\rp.\]
 We also have estimates on the geometry of~$\c L$:
 \[\la\c L\ra_{s-\frac12}+\lA \Pi_l\rA_{H^{s-2}(\c S)}+\lA n_l\rA_{H^{s-1}(\c S)}\leq C\lp1+\lA\kappa\rA_{H^{s-2}(\c S)}+\lA\kappa_l\rA_{H^{s-\frac52}(\c L)}\rp.\]
\end{lem}

\begin{proof}
The proof is standard, and we only give a sketch of it.
 It rests on the identity
 \begin{equation}
  -\Delta_{\c S}\Pi=-\c D^2\kappa+\la\Pi\ra^2\Pi-\kappa\Pi^2,
 \end{equation}
 which is proved in~\cite{ShatahZeng1} and stays valid for a hypersurface with boundary. For the case~$n\geq3$, the same identity holds for~$\c L$ in~$\c M$. 
 Then one can use~$\Phi_{\c L}$  to transfer it to an elliptic equation on some derivatives of~$\eta_l
$, and use elliptic regularity to find the regularity of~$\c L$ and the above estimates. Using again~$\Phi_{\c S}$ and elliptic estimates, this time in domains with boundary, keeping in mind the regularity of~$\c L$ as boundary data, we find the regularity of~$\c S$ and the estimates.
In dimension~$n=2$, we only need to use the identity on~$\c S$, since as remarked above, the boundary data consists only of two point whose range are bounded by the condition~$\c S\in\Lambda_*$.

\end{proof}

The advantage of this lemma is that now to control the regularity of~$\c S$, we only need to control~$\kappa$ and~$\kappa_l$, which are invariantly defined.

\subsection{Internal coordinates}
We can easily define Sobolev norms on~$\Omega$ by considering Sobolev functions as restrictions of functions Sobolev on~$\R^n$. Then
\[\lA u\rA_{H^r(\Omega)}=\inf\lB\lA U\rA_{H^r(\R^n)};u=U\rvert_{\Omega}\rB.\]
This way, the constants of Sobolev embeddings theorems are independent of the domain~$\Omega$.
Also, if $\Omega_t$ is a continuous one-parameter family of such domains, we can use that to define the classes~$C^k(H^r(\Omega_t))$ of functions $k$-differentiable in~$t$ with values in~$H^r(\Omega_t)$, simply by requiring that an extension of the function to~$\R^n$ be~$C^k$ in time with value in~$H^r(\R^n)$. It is easy to see that any other reasonable definition of~$C^k(H^r(\Omega_t))$ coincides with this one, which is in particular independent of the chosen (continuous) extension operator.

Our objective in this section is to construct a diffeomorphism from~$\Omega$ to~$\Omega_*$ with maximal regularity. As can be seen from the boundaryless case, any construction based on an affine change of variable would be only of~$H^{s}$ regularity, while we want it to be~$H^{s+\frac12}$. As we will see, the existence of this diffeomorphism is a consequence of Sobolev extension theorems in domains with edges. All of those are based on the following theorem in the model case of the quarter-space.
\begin{lem}	\label{lem:extmodel}
	For~$m\in\N_*$, the mapping~$u\mapsto\lB(f_k)_{0\leq k\leq m-1},(g_l)_{0\leq l\leq l-1}\rB$ defined by
	\[f_k=\partial_z^ku\rvert_{x=0},\;g_l=\partial_x^lu\rvert_{z=0}\]
	for~$u\in C^\infty(\overline{\R^+\times\R^+\times\R^{n-2}})$ has a unique continuous extension from~$H^m(\R^+\times\R^+\times\R^{n-2})$ onto the subspace of 
	\[\prod_{k=0}^{m-1}H^{m-k-\frac12}(\R^+\times\R^{n-2})\times\prod_{l=0}^{m-1}H^{m-l-\frac12}(\R^+\times\R^{n-2})\]
	defined by
	\begin{itemize}
		\item $\partial_x^lf_k(0)=\partial_z^kg_l(0)$, $l+k<m-1$ and
		\item $\int_0^1\lA\partial^l_xf_k(t)-\partial_z^kg_l(t)\rA_{L^2(\R^{n-2})}^2\frac{\d t}{t}<+\infty$, $l+k=m-1$.
	\end{itemize}
	It has a continuous right inverse, the extension operator.
\end{lem}
This is a trivial extension of theorem~4.3 in~\cite{MingWangDNCorner}.

With smooth local charts for the manifold with corner~$\Omega_*$, we can transfer results on the quarter-space to results close to~$\c L_*$.
One such example is the Sobolev extension theorem, used in the following Proposition on global coordinates.
\begin{prop}	\label{prop:globcoord}
	For~$\delta>0$ small enough, and~$s>1+\frac n2$, for any~$\c S\in\Lambda_*$, there exists a global diffeomorphism~$\Phi_\Omega$ from~$\Omega_*$ to~$\Omega$, restricting to~$\Phi_{\c S}$ on~$\c S_*$, and satisfying
	\[\lA\Phi_{\Omega}\rA_{H^{s}(\Omega_*)}+\lA\Phi_{\Omega}^{-1}\rA_{H^{s}(\Omega)}\leq C,\]
	with~$C$ uniform in~$\Lambda_*$.
	
	Furthermore, if~$\c S_*$ and~$\c S$ are both in~$H^s$, then
	\[\lA\Phi_{\Omega}\rA_{H^{s+\frac12}(\Omega_*)}+\lA\Phi_{\Omega}^{-1}\rA_{H^{s+\frac12}(\Omega)}\leq C\lb1+\la \c S\ra_s\rb.\]
\end{prop}
\begin{proof}
	As stated, we want the boundary value for~$\Phi_\Omega$ to be~$\Phi_{\c S}$ on~$\c S_*$. On~$\c B_*$, which is a compact hypersurface with boundary~$\c L_*$, we need it to restrict to an~$H^s$ diffeomorphism to~$\c B$, with value~$\Phi_{\c S}\rvert_{\c L_*}$ on~$\c L_*$. Since this~$\Phi_{\c S}\rvert_{\c L_*}$ is only~$H^{s-\frac12}$, we need a diffeomorphism of maximal regularity. To define such a diffeomorphism, we use the following construction. 
	Recall that we have constructed in Lemma~\ref{lem:vectcoord} a smooth vector field~$X$ whose restriction to~$\c M$ is a tangent vector field in a neighborhood of~$\c L_*$, and such that for~$p\in\c L_*$, $\Phi_{\c S}(p):=\phi(p,\eta_{\c L}(p))$ with~$\phi$ the flow of~$X$, and~$\eta_{\c L}$ an~$H^{s-\frac12}$ function on~$\c L^*$. We can extend~$X$ to a smooth tangent vector field on the whole of~$\c M$ by gluing it to the null vector field using a partition of unity. Then we can extend~$\eta_{\c L}$ to an~$H^s$ function~$\eta_{\c B}$ on~$\c B_*$, using for example a harmonic extension. Then defining~$\Phi_{\c B}(p):=\phi(p,\eta_{\c B}(p))$, with~$\phi$ still the flow of our extended vector field, we get the promised diffeomorphism. Since~$\Phi_{\c S}$ is close to identity and we have extended~$X$ by the null vector field, this diffeomorphism is close to identity.
	
	Then we want to construct~$\Phi_\Omega$ as 
	\[\Phi_\Omega:=\Id+E(\Phi_{\c S}-\Id,\Phi_{\c B}-\Id),\]
	where~$E(a,b)$ is a Sobolev extension of~$(a,b)$ in~$\Omega$. Using~$H^{s+\frac12}$ local coordinates and a partition of unity, we only have to construct such an extension in the model case of the half-plane, which is trivial, and of the quarter-space, where we want to use Lemma~\ref{lem:extmodel}. We only need to prove that we can find the~$f_k,g_l$, with~$f_0=\Phi_{\c S}-\Id$, $g_0=\Phi_{\c B}-\Id$, and with the compatibility conditions. The only condition needed between~$\Phi_{\c S}-\Id$ and~$\Phi_{\c B}-\Id$ is their equality at~$\c L_*$. Then finding the other~$f_k,g_l$ is a simple exercise.
	Continuity of~$E$ and smallness of the boundary values give us that~$\Phi_\Omega$ is a diffeomorphism satisfying the conclusions of the Proposition.
\end{proof}

We will also need product estimates, which have exactly the same form as the one on~$\c S$. Again they will be used liberally.

\subsection{Elliptic regularity}
Our next point of order is to study two operators that appear frequently in the analysis.
The first is the harmonic extension operator~$\c H$, which takes a function~$f\in H^{\sigma+\frac12}$, $0\leq \sigma\leq s-\frac12$, on~$\c S$ to the solution of
\begin{equation*} 
	\lB
	\begin{aligned}
		&\Delta\c H(f)=0\text{ in }\Omega,\\
		&\c H(f)\rvert_{\c S}=f,\\
		&\nab\nu\c H(f)\rvert_{\c B}=0.
	\end{aligned}
	\right.
\end{equation*}
The second one takes an~$H^{\sigma-1}$ function~$g$ on~$\Omega$ and an~$H^{\sigma-\frac12}$ function~$h$ on~$\c B$ to the solution~$q=:\Delta^{-1}(g,h)$ of
\begin{equation*} 
	\lB
	\begin{aligned}
		&\Delta q=g\text{ in }\Omega,\\
		& q\rvert_{\c S}=0,\\
		&\nab\nu q\rvert_{\c B}=h.
	\end{aligned}
	\right.
\end{equation*}
We would like to prove that those mapping are continuous with value in~$H^{\sigma+1}(\Omega)$. These are both particular cases of the more general problem of finding the regularity of the solution~$u$ of the problem
\begin{equation} 	\label{eq:ellprob}
	\lB
	\begin{aligned}
		&\Delta u=g\text{ in }\Omega,\\
		&u\rvert_{\c S}=f,\\
		&\nab\nu u\rvert_{\c B}=h,
	\end{aligned}
	\right.
\end{equation}
with~$(f,g,h)\in H^{\sigma+\frac12}(\c S)\times H^{\sigma-1}(\Omega)\times H^{\sigma-\frac12}(\c B)$, where~$0\leq \sigma\leq s-\frac12$, with~$h=0$ if~$\sigma<\frac12$ because it would not be defined in a strong sense.

For future reference, we give the full existence and regularity theory for those problems, and not only the a priori estimates.

In this endeavor, we are faced with two challenges. The first is technical: in order to use our estimates in the evolution problem, our constants have to be of the form~$C(1+\la\c S\ra_s)$, with~$C$ uniform in~$\Lambda_*$. To solve this, we use the global coordinates~$\Phi_{\Omega}$ defined in Proposition~\ref{prop:globcoord} to pull back the problem to~$\Omega_*$, which gives us a family of problems with coefficients bounded by a constant of the form we want. Then we prove a regularity theory for those problems, using freely the information that the surface is in~$\Lambda_*$, but using the information that it is in~$H^s$ only once. This will give us the regularity for our problems, with constants as above.

The second challenge is deeper. The domain~$\Omega_*$ has an edge, and it is well known that elliptic problems in domain with corners or edges give solution which have in general a limited regularity at the corner, whatever the smoothness of the data.
We will prove below, using variational methods, that an~$H^1$ solution always exists. If~$\sigma>0$, one would expect from the case of a regular boundary that the solution should be~$H^{\sigma+1}$. However, in general for domains with corner, the solution is not necessarily~$H^\sigma$ at the edge. To be more specific, in~$2D$ it can be decomposed between a regular~$H^\sigma$ part and an explicit sum of singularities of the form~$r^\lambda$ or~$r^\lambda\ln r$ where~$r$ is the distance to the edge, and the~$\lambda$ are a discreet set of real numbers, here of the form~$(k+1/2)\pi/\omega$ with~$\omega$ the contact angle. Therefore, the first singularity to appear, for~$\lambda=\frac{\pi}{2\omega}$, limits the regularity of the solution to~$H^{1+\frac{\pi}{2\omega}-}$. To avoid the presence of those singularities in the evolution problem, we restrict our attention to the case where~$\omega<\pi/(n+1)$, so that we can take the regularity of the surface to be~$H^s$ with~$1+\frac n2<s<{\frac{\pi}{2\omega}-\frac12}$ and have at the same time solutions of~\eqref{eq:ellprob} with the expected regularity, and enough regularity of the surface to find solutions to the Cauchy problem. Our analysis follows closely the method in~\cite{DaugeEllCor}.

Because the meaning of the problem changes from variational to classical as~$\sigma$ increases, we recast it as follows.
First, we define for each~$\c S\in\Lambda_*$ the bilinear form
\[a_{\Omega}(u,v):=\int_\Omega\nabla u\cdot\overline{\nabla v}\d x=\int_{\Omega_*}\nab *u\cdot\overline{\nab *v}\d x_*,\]
where~$\nab *$ and~$\d x_*$ are the pullback by~$\Phi$ of the gradient and the Lebesgue measure to~$\Omega_*$.  They both derive from the pullback of the Euclidean metric to~$\Omega$, giving a bounded family of~$H^{s-\mez}$ metrics on~$\Omega_*$. We also identify~$u$ and~$v$ with their pullback to keep notations simple. Those forms are well-defined on~$H^1(\Omega_*)$.

Define the variational space
\[V:=\lB v\in H^1(\Omega_*),v=0\text{ on }\c S_*\rB.\]
The family of diffeomorphisms~$\Phi_\Omega$ induce an uniformly bounded family of isomorphism between~$V$ and the space~$V_\Omega:=\lB v\in H^1(\Omega),v=0\text{ on }\c S\rB$. Therefore if~$a$ is strongly coercive on~$V_\Omega$ with a constant independent of~$\Omega$, then it will be strongly coercive on~$V$ uniformly in~$\Lambda_*$. This means that we have to prove that for any~$v\in V_\Omega$,
\[\lA v\rA_{H^1(\Omega)}\leq C\lA\nabla v\rA_{L^2(\Omega)},\]
with~$C$ depending only on~$\Lambda_*$.

To prove this, we first remark that since~$\Omega$ is Lipschitz, the space~$V$ is the adherence for the~$H^1$ norm of~$C^\infty_c(\overline{\Omega}\setminus\c S)$. Also, because~$s-\frac12>\frac{n+1}2$, the set~$\Lambda_*$ is bounded in the~$L^\infty$ topology, and therefore all the domains~$\Omega$ are contained in a band delimited by two parallel hyperplanes. The function in~$C^\infty_c(\overline{\Omega}\setminus\c S)$ are simply the restriction of smooth functions that are zero near the ``upper'' hyperplane. Because we have defined the~$H^s$ norms in~$\Omega$ as the quotient norm from~$\R^n$, our inequality is a consequence of the fact that for those smooth functions,
\[\lA v\rA_{H^1}\leq C\lA\nabla v\rA_{L^2},\]
with~$C$ depending only on the distance between those hyperplanes, which is simply the Poincaré inequality. In summary, we have proved the following lemma.
\begin{lem}
	There exists a constant~$C$ depending only on~$\Lambda_*$ such that for any~$\c S$ in~$\Lambda_*$, the form~$a_\Omega$ satisfies for any~$u\in V$
	\[\lA u\rA_{H^1(\Omega)}^2\leq Ca(u,u).\]
	Therefore by Lax-Milgram, the family~$a$ generates a bounded family of isomorphisms~$A^{*}_{\Omega}$ between~$V$ and its dual~$V'$, defined by~$(A^{*}_\Omega u)(v)=a_{\Omega}(u,v)$. 
\end{lem}

To simplify the notations, we now omit the subscript~$\Omega$ from our operators, keeping in mind that we are really dealing with a family of problems on which our estimates have to be uniform.

The meaning of our problem~\eqref{eq:ellprob} changes, as~$\sigma$ increases, from variational to classical. To be precise, we introduce the following family~$A$ of operators~$A^{(\sigma)}$.
\begin{itemize}
	\item For~$0\leq \sigma\leq\frac12$, we take
	\[A^{(\sigma)}:H^{\sigma+1}(\Omega_*)\rightarrow H^{\sigma+\frac12}(\c S_*)\times(V^{1-\sigma})',\]
	where~$V^{1-\sigma}:=\lB v\in H^{1-\sigma}(\Omega_*);v\rvert_{\c S_*}=0\rB$. It is defined by
	\[A^{(\sigma)}u=(u\rvert_{\c S_*},g),\]
	where for any~$v\in V^{1-\sigma}$,
	\[g(v)=a(u,v).\]
	\item For~$\frac12< \sigma<1$, we take
	\[A^{(\sigma)}:H^{\sigma+1}(\Omega_*)\rightarrow H^{\sigma+\frac12}(\c S_*)\times(H^{\sigma-1}(\Omega_*))\times H^{\sigma-\frac12}(\c B_*),\]
	defined by
	\[A^{(\sigma)}u=(u\rvert_{\c S_*},g,(\nab \nu)_*u\rvert_{\c B_*}),\]
	where for any~$v\in H^{1-\sigma}(\Omega_*)$,
	\[g(v)=a(u,v)-\int_{\c B_*}(\nab \nu)_*u\overline v,\]
	with~$(\nab \nu)_*u$ the pullback to~$\c B^*$ of~$\nab \nu u$ on~$\c B$.
	Here we have identified~$H^{\sigma-1}$ as the dual to~$H^{1-\sigma}$ since~$0<1-\sigma<\frac12$. 
	\item For~$1\leq \sigma\leq s-\frac12$, we take
	\[A^{(\sigma)}:H^{\sigma+1}(\Omega_*)\rightarrow H^{\sigma+\frac12}(\c S_*)\times H^{\sigma-1}(\Omega_*)\times H^{\sigma-\frac12}(\c B_*),\]
	defined by
	\[A^{(\sigma)}u=(u\rvert_{\c S_*},-\Delta_*u,(\nab \nu)_*u\rvert_{\c B_*}),\]
	where~$\Delta_*$ is the Laplace operator for the pulled-back metric.
\end{itemize}

We remark that this family of operators correspond more properly to~$-\Delta$, which of course does not change anything.

Because of Green's identity on the domain~$\Omega$, there holds
\[a(u,v)=\int_{\Omega_* }(-\Delta_*)u\overline{v}+\int_{\c S_*}(\nab \nu)_*u\overline v+\int_{\c B_*}(\nab \nu)_*u\overline v\]
for regular functions.

The expression
\[\int_{\Omega_*}\nab *u\cdot\overline{\nab *v}\d x_*\]
makes sense when the integral is interpreted as a duality product in~$H^{\sigma}\times H^{-\sigma}$, for~$0\leq\sigma\leq\frac12$. Therefore if~$0\leq \sigma\leq\frac12$,  for~$u\in H^{1+\sigma}$, $a(u,.)$ can be thought of as a linear form on~$V^{1-\sigma}$. The definition is to be interpreted in this sense.

If~$v$ is zero on~$\c S$ and the functions are regular enough, there holds
\[a(u,v)-\int_{\c B_*}(\nab \nu)_*u\overline v=\int_{\Omega_* }(-\Delta_*)u\overline v,\]
and again if we take the integral as a duality product, for~$\frac12<\sigma\leq1$ the right-hand side makes sense for~$u\in H^{1+\sigma}$ as a linear form on~$v\in H^{1-\sigma}$. It is again in this sense that the definition is to be interpreted.

For~$\sigma=1$, the Green formula tells us the classical and variational formulations coincide.

\begin{prop}	\label{prop:varsol}
	The operator~$A^{(0)}$ is an isomorphism between~$H^1(\Omega_*)$ and~$H^{\frac12}(\c S_*)\times(V)' $, whose inverse has bounded norm as~$\c S$ varies in~$\Lambda_*$.
\end{prop}
\begin{proof}
	Being given~$(f,g)\in H^{\frac12}(\c S_*)\times(V)'$, We want to find~$u\in H^{\sigma+1}(\Omega_*)$ such that~$f=u\rvert_{\c S_*}$ and~$g(v)=a(u,v)$ for~$v\in V$. 
	We consider a Sobolev extension~$f\mapsto\tilde f$ from~$H^{\frac12}(\c S_*)$ to~$H^1(\Omega_*)$. It exists by a construction similar to the one in the proof of Proposition~\ref{prop:globcoord}. Then we use the strong coercivity of~$a$ to find~$\tilde{u}\in V$ such that
	\[a(\tilde u,v)=g(v)-a(\tilde f,v),\]
	and we set~$u=\tilde u+\tilde f$. The fact that the constant associated to this construction is uniform in~$\Lambda_*$ comes from the uniformity of the constant in the coercivity of~$a$.
\end{proof}

We link those different formulations using the following embeddings~$I_{\sigma,\sigma'}$ of the target of~$A^{(\sigma')}$ into the target of~$A^{(\sigma)}$, for~$\sigma<\sigma'$.
\begin{itemize}
	\item If~$0\leq \sigma<\sigma'<\frac12$, they are the canonical embeddings of the space~$H^{\sigma+\frac12}(\c S_*)\times(V^{1-\sigma})'$ in~$H^{\sigma'+\frac12}(\c S_*)\times(V^{1-\sigma'})'$.
	\item If~$0\leq \sigma<\frac12<\sigma'$, we write~$I_{\sigma,\sigma'}(f,g,h)=(f,g')$ with
	      \[g'(v)=g(v)+\int_{\c B_*}h\overline v\]
	      which define an embedding from~$H^{\sigma+\frac12}(\c S_*)\times(H^{1-\sigma}(\Omega_*))'\times H^{\sigma-\frac12}(\c B_*)$ to~$H^{\sigma+\frac12}(\c S_*)\times(V^{1-\sigma})'$.      
	\item If~$\frac12<\sigma<\sigma'$,
	      again we take the trivial embedding.
\end{itemize}
We remark that the definition does not depend on~$a$, and is therefore the same for all~$\c S\in\Lambda_*$.
The following is immediate.
\begin{lem}
\begin{enumerate}
	\item For~$0\leq \sigma<\sigma'<\sigma''$, 
	      \[I_{\sigma,\sigma'}\circ I_{\sigma',\sigma''}=I_{\sigma,\sigma''}.\]
	\item For~$0\leq \sigma<\sigma'$,
	      \[A^{(\sigma)}\rvert_{H^{1+\sigma'}}=I_{\sigma,\sigma'}\circ A^{(\sigma')}.\]
	\item For~$0\leq \sigma<\sigma'$, $I_{\sigma,\sigma'}$ is compact.
\end{enumerate}
\end{lem}

As a consequence of this and of Proposition~\ref{prop:varsol}, being given~$(f,g)\in H^{\sigma+\frac12}(\c S_*)\times(V^{1-\sigma})'$ for~$0\leq \sigma<\frac12$, or~$(f,g,h)\in H^{\sigma+\frac12}(\c S_*)\times H^{\sigma-1}(\Omega_*)\times H^{\sigma-\frac12}(\c B_*)$ for~$\frac12<\sigma\leq s-\frac12$, we can always define an~$H^1$ variational solution~$u$ by
\[u=(A^{(0)})^{-1}I_{0,\sigma}(f,g,h).\]

Since we have included transversality in our definition of~$\Lambda_*$, for any~$\c S$ in it and any point on its water line~$\c L$, we can define a contact angle as the angle between the inward normal vector to~$\c L$ in~$\c B$ and the inward normal vector to~$\c L$ in~$\c S$. This defines a continuous function on~$\c L$ because~$s-\frac12>\frac{n+1}2$ and~$n\geq2$. Therefore by taking~$\delta$ small enough, which shrinks~$\Lambda_*$, we can assume that all the angles of all the surfaces in~$\Lambda_*$ lie in an interval~$[\underline\omega,\overline\omega]$ with~$0<\underline\omega$ and~$\overline\omega<2\pi$.
Our aim is then to prove the following theorem.
\begin{thm}	\label{thm:ellreg}
	For any~$\c S$ in~$\Lambda_*$, for any~$0\leq \sigma<\min(s-1,\frac{\pi}{2\overline\omega})$, the operator~$A^{(\sigma)}_{\Omega}$ is an isomorphism.
	The norm of the inverse is uniformly bounded in~$\Lambda_*$.
	
	If~$\c S$ is also in~$H^s$, and if~$\sigma<\min(s-\frac12,\frac{\pi}{2\overline\omega}),$ then the operator~$A^{(\sigma)}_{\Omega}$ is still an isomorphism,
	and the norm of the inverse is bounded by~$C(1+\la\c S\ra_s)$, with~$C$ uniform in~$\Lambda_*$.
\end{thm}
Because we already now that a variational solution~$u$ exists, and because~$A^{(0)}\rvert_{H^{1+\sigma'}}=I_{0,\sigma}\circ A^{(\sigma)}$, the statement is really on the regularity of this variational solution. 

Since~$\overline{\Omega_*}$ is compact, the regularity of the solution is equivalent to its regularity in a neighborhood of each point. Regularity near interior points and near regular points of the boundary are classical, because~$\Delta$ is an elliptic operator, so that~$\Delta_*$ is also elliptic. In fact, the maximal angle does not limit the regularity there. The dependence on the constant comes from the bootstrap nature of the estimates: we prove the regularity at level~$H^\sigma$ by assuming it at level~$H^{\sigma-\frac12}$, so that we only need to use the information that~$\c S$ is~$H^s$ once.

We concentrate on the case of the neighborhood of a point~$x\in\c L$. The method rests on the analysis of a constant-coefficients model operator, which we think of as ``frozen'' at the point~$x$, the regularity of which will imply the regularity of the original problem. 

\paragraph{Model operator}
The model operator is constructed as follows. We start with the form~$a$ on~$\Omega_*$. It can be written
\[a(u,v)=\sum_{i,j=1}^n\int_{\Omega_*}\alpha_{i,j}\partial_iu\overline{\partial_jv},\]
where the~$\alpha_{i,j}$ are~$H^{s-\frac12}$ functions in~$\Omega_*$. Because~$s-\frac12>\frac{n+1}2$, they are continuous functions and therefore make sense at~$x$. We can then consider the constant-coefficients form~$p$ on~$\R^{n-2}\times\Gamma_*$, where~$\Gamma_*$ is the conical sector of the plane of summit~$0$ and angle~$\omega_*$, the contact angle at~$x$, defined by the formula
\[p(u,v)=\sum_{i,j=1}^n\int_{\R^{n-2}\times\Gamma_*}\alpha_{i,j}(x)\partial_iu\overline{\partial_jv}.\]
Then we further reduce the problem to the form~$l$ defined on~$\Gamma_*$ by
\[l(u,v)=\sum_{i,j=1}^2\int_{\Gamma_*}\alpha_{i,j}(x)\partial_iu\overline{\partial_jv},\]
where we have only kept the derivatives corresponding to the variables in~$\Gamma_*$. This form also satisfies a Green formula, and we can attach it to a family of problems~$L^{(\sigma)}$ in the exact same way we have used for~$A$. The constant-coefficients differential operator~$L$ it is attached to is simply obtained by freezing the principal part of~$\Delta_*$ at~$x$ and replacing all derivatives in the part tangential to~$\c L_*$ by~$0$, and the Neumann condition is for the normal vector with constant coefficients frozen at~$x$ and projected onto~$\Gamma_*$. 

To further simplify the study of those problems, we notice that this definition is really an invariant notion. The reason is that if we have a local~$H^{s}$ diffeomorphism near~$x$ to another neighborhood of a point in the transversal intersection of two hypersurfaces, its differential at~$x$ makes sense again because~$s-1>\frac{n}2$, and can be interpreted as a linear diffeomorphism between~$\R^{n-2}\times\Gamma_*$ and the corresponding angular sector~$\R^{n-2}\times\Gamma'$, sending~$\Gamma_*$ isomorphically onto~$\Gamma'$. It is then easy to see that the restriction of this linear isomorphism to~$\Gamma_*$ sends~$l$ to the constant-coefficients form obtained by sending~$a$ to the new domain, and freezing the coefficients as above. Applying this to the diffeomorphism~$\Phi$ between~$\Omega_*$ and~$\Omega$, we see that the study of the~$L^{(\sigma)}$ is equivalent to the study of the problems derived from the form
\[\int_{\Gamma}\nabla u\cdot\overline{\nabla v}\d z\]
in the sector~$\Gamma$ whose angle~$\omega$ is the original contact angle in~$\Omega$.
It corresponds to the constant coefficients Laplace operator, with Dirichlet condition on one edge and Neumann condition on the other.
Because we want our constants to be uniform in~$\Lambda_*$, we will study those operators in the fixed domain~$\Gamma_*$.  We will however use the version in~$\Gamma$ to find a particular algebraic condition. Remark that the coefficients of the form~$l$ are a fixed number of constants inhabiting a compact set as~$\c S$ varies in~$\Lambda_*$.
 
We denote by~$B_*$ the bottom edge and by~$S_*$ the surface edge of~$\Gamma_*$.
Our aim  is to prove the following Proposition.
\begin{prop}	\label{prop:modelreg}
	For any~$\c S$ in~$\Lambda_*$, for any~$0\leq \sigma<\frac{\pi}{2\omega}$, if there exists~$u\in H^1(\Gamma_*)$ compactly supported and
	\begin{enumerate}
		\item $(f,g)\in H^{\sigma+\frac12}(S_*)\times(V^{1-\sigma})'$ for~$0\leq \sigma<\frac12$,
		\item $(f,g,h)\in H^{\sigma+\frac12}(S_*)\times(H^{1-\sigma}(\Gamma_*))'\times H^{\sigma-\frac12}(B_*)$ for~$\frac12<\sigma<1$,
		\item	$(f,g,h)\in H^{\sigma+\frac12}\times H^{\sigma-1}(\Gamma_*)\times H^{\sigma-\frac12}(B_*)$ for~$1\leq \sigma$,
	\end{enumerate}
	such that
	\[L^{(0)}u=I_{0,\sigma}(f,g)\quad(\text{resp. }L^{(0)}u=I_{0,\sigma}(f,g,h)),\]
	then~$u\in H^{\sigma+1}(\Gamma_*)$, with constants of the good form, depending on the support of~$u$.
\end{prop}
Here we have abused notation to write~$V^{1-\sigma}=\lB v\in H^{1-r}(\Gamma_*),v\rvert_{S_*}=0\rB$ and the~$I_{\sigma,\sigma'}$ have the obvious definitions. Their properties are the same, except that since~$\Gamma_*$ is not compact, they are only compact for functions with compact support.

The reason for the index~$\frac{\pi}{2\omega}$ is algebraic, and comes from the following. We let~$(r,\theta)$ be the polar coordinates in~$\Gamma$. 
For any~$\lambda$ in~$\C$, and any~$\sigma\in R$, we define
\begin{equation}	\label{eq:defsing}
S^{\lambda,\sigma}:=\lB v=\sum_{q=0}^Qr^\lambda\log^qrv_q(\theta);v_q\in H^{1+\sigma}([0,\omega_*])\rB.
\end{equation}
We say that such a function is a polynomial if it is polynomial in the coordinates~$z_1,z_2$. This is only possible for~$\lambda\in\N^*$, or if~$v=0$.
We can define~$L$ on~$S^{\lambda,\sigma}$ by testing against compactly supported function for the duality of~$V^{1-\sigma}$ or~$H^{1-\sigma}$, as is easily seen.
Then we say that for~$v\in S^{\lambda,\sigma}$, $Lv$ is polynomial if there exists polynomials~$(f,g)$ (resp.~$(f,g,h)$) such that~$L^{(0)}v=I_{0,\sigma}(f,g)$
(resp.~$L^{(0)}v=I_{0,\sigma}(f,g,h)$). If~$\lambda$ is not in~$\N^*$, then~$(f,g,h)$ is zero. At last, $L$ is said to be injective modulo polynomials if~$v\in S^{\lambda,\sigma}$ is polynomial as soon as~$(f,g,h)$ are polynomials. Again if~$\lambda$ is not an integer, this is just injectivity. 

\begin{lem}	\label{lem:singvap}
	Let~$\sigma>0$. Then~$L$ is injective modulo polynomials exactly for~$\lambda\neq(k+\frac12)\frac{\pi}{\omega}$, with~$k\in\Z$.
\end{lem}

\begin{proof}
	It is easily seen that this notion is invariant by linear diffeomorphism, and therefore it is sufficient to check it on the euclidean Laplace operator in~$\Gamma$.
	Because of the classical ellipticity of~$\Delta$ in the~$\theta$ variable, this notion is independent of the regularity~$\sigma$. We can therefore check it for smooth~$v_q$.
	
	Then injectivity modulo polynomials can be computed to be a cascade of ODEs on the~$v_q$ of the form
	\[\partial_\theta^2v_q+\lambda^2u_q+2\lambda(q+1)v_{q+1}+(q+2)(q+1)v_{q+2}=0\]
	for~$q\geq1$, with~$v_{Q+1}=v_{Q+2}=0$. Then starting from~$q=Q$, we show that if~$\lambda\neq(k+\frac12)\frac{\pi}{\omega}$, all the~$v_q$ are~$0$ except~$v_0$, which is such that~$v$ is polynomial.
	If~$\lambda=(k+\frac12)\frac{\pi}{\omega}$, the fact that~$\partial_\theta^2+\lambda^2$ is not injective yields immediately a counter-example.
\end{proof}

The rest of the proof is based on the use of the Mellin transform, for which the properties in Appendix~A to~\cite{DaugeEllCor} will suffice. We only recall its definition,
\[\c M[u](\lambda)=\int_{\R} r^{-\lambda}u\frac{\d r}{r},\]
and the following proposition:
\begin{prop}
 \begin{itemize}
  \item If~$u\in H^\beta(\Gamma_*)$, with~$\beta<1$, and~$u$ is compactly supported, then~$\c M[u]$ is defined up to ~$\Re\lambda=\beta-1$ and analytic in~$\Re\lambda<\beta-1$ with values in~$H^\beta([0,\omega_*])$.
  \item If~$u\in H^\beta(\Gamma_*)$, with~$\beta\geq1$, and~$u$ is compactly supported, then~$\c M[u]$ is holomorphic in~$\Re\lambda<0$ with values in~$H^\beta([0,\omega_*])$. It can be meromorphically extended up to~$\Re\lambda<\beta-1$. Let~$U$ denote that extension. For~$\Re\lambda\in]k,k+1[$, $k\in\N$, $U$ coincides with~$\c M[u-P_ku]$, where~$P_ku$ denotes the Taylor series of order~$k$ of~$u$ at~$0$. Finally, the poles of~$U$ are simple and lie in~$k\in\N$ with~$k<\beta-1$, and we have
  \[\mathrm{Res}_{k=\lambda}r^\lambda U(\lambda)=P_{k-1}u-P_ku.\]
 \end{itemize}
\end{prop}

There is an inversion formula,
\[u(z)=\frac1{2\pi}\int_{\R} r^{-\eta+i\zeta}\c M[u](-\eta+i\zeta)\d\zeta,\]
valid on any vertical line in the domain of holomorphy of~$\c M[u]$. At last, one can measure the~$H^\beta(\Gamma_*)$ regularity of~$u$ from the~$L^2$ norm in~$\zeta$ of the~$H^\beta([0,\omega_*],\la\zeta\ra+1)$ norm in~$\eta$ of~$\c M[u](-\eta+i\zeta)$. Here and in the following, $H^\beta([0,\omega_*],\rho)$ is the~$H^\beta$ space endowed with the norm
\[\lA f\rA^2_{\beta,\rho}=\rho^{2\beta}\lA f\rA_{L^2}^2+\lA D^{\lfloor\beta\rfloor}f\rA_{H^{\beta-\lfloor\beta\rfloor}}^2.\]

The interest of the Mellin transform is that it changes~$r\partial_r$ into~$\lambda$, so that the problem
 \begin{equation*} 	
	\lB
	\begin{aligned}
		&\Delta u=g\text{ in }\Gamma,\\
		&u\rvert_{\theta=0}=f,\\
		&\nab\nu u\rvert_{\theta=\omega}=h,
	\end{aligned}
	\right.
\end{equation*}
would become, through multiplication by the appropriate power of~$r$ and the Mellin transform,
\begin{equation*} 	
	\lB
	\begin{aligned}
		&(\partial_\theta^2+\lambda^2) U=\c M[r^2g]\text{ in }\Omega,\\
		&U\rvert_{\theta=0}=\c M[f],\\
		&\partial_\theta u\rvert_{\theta=\omega}=\c M[rh].
	\end{aligned}
	\right.
\end{equation*}
Of course, we need to work in~$\Omega_*$ and with the associated form~$l$.

To simplify notations, we concentrate on the case~$s-1\geq\sigma>\frac12$, the proof of the other cases being similar. 

So assume we are in the hypotheses of Proposition~\ref{prop:modelreg}. Then in Mellin, we give the following definitions.
\begin{itemize}
	\item The Mellin transform of~$u$ is~$U(\lambda)$, which is holomorphic in~$\Re \lambda<0$ and defined up to~$\Re \lambda=0$, with values in~$H^1([0,\omega_*])$.
	\item The Mellin transform of~$(f,r^2g,rh)$ is~$(F,G,H)(\lambda)$, which is meromorphic on~$\Re \lambda<\sigma$, defined up to~$\Re\lambda=\sigma$, with values in~$\R\times(H^{1-\sigma}([0,\omega_*]))'\times\R$ (or~$\R\times H^{\sigma-1}([0,\omega_*])\times\R$). Its poles are concentrated at the non-negative integer values of~$\lambda$. 
	\item The operator~$L^{(0)}$ becomes a holomorphic family of operators~$\c L^{(0)}(\lambda):H^1([0,\omega_*])\rightarrow\R\times (V([0,\omega_*]))'$ defined by
	\[\lp \c L^{(0)}(\lambda)v\rp(w):=(v(\omega_*),l(\lambda)(v,w)),\]
	with~$l(\lambda)$ a holomorphic family of forms on~$V([0,\omega_*])$ with constant coefficients, bounded as~$\c S$ varies in~$\Lambda_*$.
	\item The~$I_{0,\sigma}$ are transformed in a holomorphic family of compact operators~$\c I_{0,\sigma}(\lambda)$.
\end{itemize}

Now the equation on~$u$ becomes
\[\c L^{(0)}(\lambda)U(\lambda)=\c I_{0,\sigma}(\lambda)(F(\lambda),G(\lambda),H(\lambda)).\]
It is valid on the common domain of holomorphy of those functions, which is~$\Re \lambda<0$.

\begin{lem}	 \label{lem:largelambda}
	For any real numbers~$\alpha<\beta$, there is a constant~$C_{\alpha,\beta}$ uniform in~$\Lambda_*$ such that for any~$\alpha<\Re\lambda<\beta$ with~$\la\Im\lambda\ra>>1$, again uniformly in~$\Lambda_*$, $\c L^{(0)}(\lambda)$ is invertible and satisfies
	\[\lA\c L^{(0)}(\lambda)^{-1}(F,G)\rA_{H^1}\leq C_{\alpha,\beta}\lb\la F\ra+\lA G\rA_{V'}\rb.\]
	The family~$\c L^{(0)}(\lambda)^{-1}$ is meromorphic in~$\C$.
\end{lem}

\begin{proof}
	The first step is to deduce the coercivity of~$l$ from the one of~$a$. 
	We can find an~$H^{s+\frac12}$ diffeomorphism sending a neighborhood of our point~$x$ to a neighborhood of~$0$ in~$\R^{n-2}\times\Gamma_*$ such that the frozen forms it produce are the~$p$ and~$l$ already defined. It suffices to compose our local coordinates map with the inverse of its linearized at~$x$.
	
	Then the coefficients of this form~$a$ are regular enough that
	\[\la p(u,u)-a(u,u)\ra\leq C\lb\rho\lA u\rA^2_{H^1}+\lA u\rA_{L^2}^2\rb\]
	for any~$u\in V$ with support in~$B(0,\rho)$,$\rho\leq1$. Combined with the coercivity of~$a$, this yields
	\[\lA u\rA^2_{H^1}\leq C\lB p(u,u)+\lA u\rA^2_{L^2}\rB\]
	if~$\rho$ is small enough. Applying this to~$v=u(\frac1\rho\cdot)$ and letting~$\rho$ go to~$0$ gives
	\[\lA v\rA^2_{H^1}\leq C\lB p(v,v)+\lA v\rA^2_{L^2}\rB\]
	for~$v$ compactly supported, and thus by density for any~$v$. Then
	applying this to~$v(y,z)=\psi(y)w(z)$ with~$\psi$ compactly supported, and equal to one on the unit ball of~$\R^{n-2}$, we get
	\[\lA w\rA^2_{H^1}\leq C\lB l(w,w)+\lA w\rA_{H^1}\lA w\rA_{L^2}\rB.\]
	Applying this result to~$w(z)=\chi(r)r^\lambda u(\theta)$, for~$\chi$ compactly supported, equal to~$1$ near~$r=1$ and to~$0$ near~$r=0$, and for~$\alpha<\Re\lambda<\beta$, this yields
	\[\lA u\rA^2_{H^1(\la\lambda\ra)}\leq C_{\alpha,\beta}\lB\Re l(\lambda)(u,u)+\lA u\rA_{H^1(\la\lambda\ra)}\lA u\rA_{L^2}\rB.\]
	Remarking that~$\lA u\rA_{L^2}\leq C\lambda^{-1}\lA u\rA_{H^1(\la\lambda\ra)}$, we get for~$\lambda$ big enough
	\[\lA u\rA^2_{H^1(\la\lambda\ra)}\leq C_{\alpha,\beta}\lB\Re l(\lambda)(u,u)\rB.\]
	This holds for~$u\in V([0,\omega_*])$ and is the required coercivity.
	This is enough to prove the inversibility of~$\c L(\lambda)$ and the accompanying estimates. 
	
	For the meromorphy, we start by observing that for~$\lambda,\lambda'\in\C$, the operator~$\c L^{(0)}(\lambda)-\c L^{(0)}(\lambda')$ is compact. Then if~$\lambda_0$ is such that~$\c L^{(0)}(\lambda_0)$ is invertible,  we can write
	\[\c L^{(0)}(\lambda)=\c L^{(0)}(\lambda_0)\circ(I+\c Z(\lambda))\]
	with
	\[\c Z(\lambda)=\c L^{(0)}(\lambda_0)^{-1}\circ(\c L^{(0)}(\lambda)-\c L^{(0)}(\lambda_0)).\]
	Thus~$\c Z(\lambda)$ is a holomorphic family of compact operators, with~$I+\c Z(\lambda_0)=I$ invertible. This implies the meromorphy of the family of index~$0$ Fredholm operators~$\c L^{(0)}(\lambda)$, and therefore the meromorphy of~$\c L^{(0)}(\lambda)^{-1}$. 
\end{proof}

Then we need to prove the regularity of the family~$\c L$.
\begin{lem}
	For any~$\Re\lambda\in[0,\sigma]$, if~$U\in H^1([0,\omega_*])$ is such that
	\[L^{(0)}(\lambda)U=\c I_{0,\sigma}(\lambda)(F,G,H),\]
	then~$U\in H^{\sigma+1}([0,\omega_*])$, and
	\[\lA U\rA_{H^{\sigma+1}}\leq C\lb\la F\ra+\lA G\rA_{(H^{1-\sigma})'}+\la H\ra\rb\]
	with~$C$ uniform in~$\lambda$ and in~$\Lambda_*$.
\end{lem}
\begin{proof}
	Because the~$L$ are constant coefficients, their coefficients when~$\c S$ varies in~$\Lambda$ inhabit a compact set. Therefore it is sufficient to prove this for a fixed~$\c S\in\Lambda_*$, with a constant that is an increasing function of the supremum of the coefficients.
	The operator family~$L$ is elliptic, and with constant coefficients. The classical elliptic regularity theory can be applied far from~$0$, near any interior and boundary point. The constant depends on the coefficients continuously. Then applying this regularity to~$w(z)=\chi(r)r^\lambda u(\theta)$ as above, this function being compactly supported away from~$0$, we immediately find the announced regularity. 
\end{proof}

Then we want to define~$U(\lambda)$, which only made sense for~$\Re\lambda\leq0$, in the half-plane~$\Re\lambda\leq\sigma$.  The formula
\begin{equation}	\label{eq:extmero}
	U(\lambda)=\c L^{(\sigma)}(\lambda)^{-1}(F(\lambda),G(\lambda),H(\lambda))
\end{equation}
agrees with the definition of~$U(\lambda)$ in the left half-plane, because
\[\c L^{(0)}\rvert_{H^{1+\sigma}}=I_{0,\sigma}\circ \c L^{(\sigma)}.\]
Also we have already seen that~$\c L^{(0)}(\lambda)^{-1}$ is meromorphic. Its poles are the places where~$\c L^{(0)}(\lambda)$ fails to be injective. It is immediate that the injectivity of~$\c L^{(0)}$ on~$H^1([0,\omega_*])$ is equivalent to the injectivity of~$L$ on~$S^{\lambda,0}$, the space defined in~\eqref{eq:defsing}. Because of Lemma~\ref{lem:singvap}, we know that for~$\sigma<\frac{\pi}{2\omega}$, this injectivity is true for any~$\lambda$, except for the integers, where only injectivity modulo polynomials holds. Therefore the only possible poles of~$\c L^{(0)}(\lambda)^{-1}$ when~$\sigma$ is in the range of Proposition~\ref{prop:modelreg} are the integers between~$0$ and~$\sigma$.

Because of the preceding Lemma, $\c L^{(\sigma)}(\lambda)^{-1}$ is also meromorphic with the same poles. Then because~$F$, $G$ and~$H$ where also meromorphic with integer poles, \eqref{eq:extmero} define a meromorphic extension of~$U$ to~$\Re\lambda<\sigma$, with integer poles.

We want to take the inverse Mellin transform of~$U(\lambda)$ along~$\Re\lambda=\sigma$. However, because of the possible presence of a pole at~$\sigma$ when~$\sigma\in\N$, we need to be careful.

First, take~$\sigma'\leq\sigma$ that is not an integer. If already~$\sigma$ is not an integer, we can take~$\sigma'=\sigma$.
Then~$U(\lambda)$ is meromorphic on~$\Re\lambda\leq\sigma'$. 

Taking the inverse Mellin transform along~$\Re\lambda=\sigma'$, we find a function~$u_0\in H^{\sigma'+1}_0(\Gamma_*)$, with norm bounded by
\[\int_{\Re\lambda=\sigma'}\lA U(\lambda)\rA^2_{\sigma'+1,\la\lambda\ra}\d\lambda.\]
However 
\[\lA U(\lambda)\rA^2_{\sigma'+1,\la\lambda\ra}\leq C\lb\la F(\lambda)\ra+\lA G(\lambda)\rA_{\sigma'-1,\la\lambda\ra}+\la H(\lambda)\ra\rb,\]
the value for large~$\lambda$ coming from Lemma~\ref{lem:largelambda} and the elliptic regularity, while for the other values of~$\lambda$ we use the qualitative information that~$\c L^{(s)}$ is invertible to find the equality with constants~$C_\lambda$ and then we take the supremum of those constants, because~$\lambda$ is in a compact set. This constant is also uniform in~$\Lambda_*$ because the coefficients of the problems are in a compact set.
This means that
\[\lA u_0\rA_{\sigma'+1}\leq C\lb \lA f\rA_{\sigma+\frac12}+\lA g\rA_{\sigma-1}+\lA h\rA_{\sigma-\frac12}\rb.\]

Now because of the Cauchy formula between~$\Re\lambda=0$ and~$\Re\lambda=\sigma'$, we find
\[u-u_0=2i\pi\sum_{\mu=0}^{\lfloor\sigma'\rfloor}\mathrm{Res}_{\lambda=\mu}r^\lambda\c L^{(\sigma')}(\lambda)^{-1}(F(\lambda),G(\lambda),H(\lambda)).\]

Then writing~$u_\mu$ for the~$\mu$th residue, we see that for~$\chi$ a cutoff at ~$0$, for~$\mu>0$, $\chi u_\mu\in V$. Since~$u\in V$, we deduce immediately that~$\chi u_0\in V$ and therefore~$u_0$ is a polynomial. We want to prove by recurrence that all the~$u_k$ are polynomials. By subtracting the ones already known to be polynomials, it is sufficient to assume that
\[\forall \mu\leq k-1,\mathrm{Res}_{\mu=\lambda}r^\lambda U(\lambda)=0.\]
Then comparing~$l(u-u_k,v)$ and~$l(u,v)$ for~$v\in V$ shows that~$a(u_k,v)$ is a polynomial. To conclude, injectivity modulo polynomials shows that~$u_k$ is a polynomial. The uniformity of the constants in~$\Lambda_*$ is immediate.
This ends the proof of Proposition~\ref{prop:modelreg} in the case where~$\sigma\notin\N$.

If~$\sigma\in\N$, we just proved that the regularity holds for~$\sigma'<\sigma$. Keeping separate the regular part~$u_0\in H^{\sigma'+1}$ and the residues (which are polynomials), we see that we only need to prove that~$\chi u_0\in H^{\sigma+1}$ with~$\chi$ a cut-off at~$0$. Thus we only need to show that 
\[\forall\alpha, \la\alpha\ra=\sigma+1, D^\alpha u_0\in L^2(\Gamma_*).\]
For this we want to extend 
\[w(\lambda)=(\c D^\alpha(\lambda)U(\lambda))_{\la\alpha\ra=\sigma+1}\]
up to~$\Re\lambda=\sigma$, with the punctual convergence
\[w(\sigma'+i\cdot)\rightarrow w(\sigma+i\cdot)\text{ in }L^2(\R\times[0,\omega_*])\cap L^{2,\sigma}(\R;L^2([0,\omega_*])),\]
where~$\c D^{\alpha}(\lambda)$ is the Mellin transform of~$D^\alpha$ and
\[L^{2,\sigma}(\R;H)):=\lB u; (1+\la\cdot\ra)^\sigma\lA u(\cdot)\rA_H\in L^2(\R)\rB.\]
Then passing to the limit~$\Re\sigma'\rightarrow\Re\sigma$ in the Cauchy formula would give us the result. But the usual limit case for the Mellin transform (again see~\cite{DaugeEllCor}) gives the desired result. 
This concludes the proof of Proposition~\ref{prop:modelreg}.

\paragraph{Full operator.}
The objective is now to prove the regularity of the full problem starting from this model case.
Recall that from the full form~$a$ we have first constructed a constant coefficients form~$p$ on~$\R^{n-2}\times\Gamma_*$ be freezing the coefficients at~$0$. We can associate to this form~$p$ a series of problems~$P$ as we have done for~$a$ and~$l$. The first objective is to derive the following regularity theory for~$P$.
\begin{prop}	\label{prop:frozenreg}
	For any~$\c S$ in~$\Lambda_*$, for any~$0\leq \sigma<\frac{\pi}{2\omega}$, if there exists~$u\in H^1(\R^{n-2}\times\Gamma_*)$ and
	\begin{enumerate}
		\item $(f,g)\in H^{\sigma+\frac12}(\R^{n-2}\times S_*)\times(V^{1-\sigma})'$ for~$0\leq \sigma<\frac12$,
		\item $(f,g,h)\in H^{\sigma+\frac12}(\R^{n-2}\times S_*)\times(H^{1-\sigma}(\R^{n-2}\times\Gamma_*))'\times H^{\sigma-\frac12}(\R^{n-2}\times B_*)$ for~$\frac12<\sigma<1$,
		\item	$(f,g,h)\in H^{\sigma+\frac12}(\R^{n-2}\times S_*)\times H^{\sigma-1}(\R^{n-2}\times\Gamma_*)\times H^{\sigma-\frac12}(\R^{n-2}\times B_*)$ for~$1\leq \sigma$,
	\end{enumerate}
	such that
	\[P^{(0)}u=I_{0,\sigma}(f,g)\quad(\text{resp. }P^{(0)}u=I_{0,\sigma}(f,g,h)),\]
	then~$u\in H^{\sigma+1}(\R^{n-2}\times\Gamma_*)$ near~$0$, with constant of the good form.
\end{prop}
We have again used the obvious definition for~$V$ and the~$I_{0,\sigma}$.
This statement is different from the one for~$L$ only if~$n\geq3$, so that their is actually a transverse direction~$y\in\R^{n-2}$.
\begin{proof}
The spirit of the proof is to first go to Fourier in the unbounded variable~$y\in\R^{n-2}$, and then use homogeneity to reduce to the dual variable~$\eta$ in the sphere~$S^{n-3}$. Then for such an~$\eta$, the regularity far from~$0$ will be a simple consequence of the classical regularity theory, while the regularity near~$0$ will come from the one of~$L$. 

Accordingly, for~$\eta\in S^{n-3}$, let~$p(\eta)$ be the form deduced from~$p$ by going to Fourier in~$y$, or equivalently replacing the~$\partial^\alpha_y$ derivatives with multiplication by~$i\eta^\alpha$. Let~$P(\eta)$ be the associated family of problems. Assume we have~$u\in V(\Gamma_*)$ and~$(f,g,h)\in H^{\sigma+\frac12}(S_*)\times(H^{1-\sigma}(\Gamma_*))'\times H^{\sigma-\frac12}(B_*)$ such that 
\[P^{(0)}(\eta)u=I_{0,\sigma}(f,g,h)).\]
We want to prove that~$u\in H^{\sigma+1}(\Gamma_*)$, with the associated constant uniform in~$\Lambda_*$. Here~$u$ need not have compact support, which is important for what follows.
We can decompose between the regularity at points far from~$0$, and at~$0$.
The regularity at~$0$ is an immediate consequence of the regularity of~$L$ and the easily established compactness of~$P(\eta)-L$ on compactly supported functions, which comes from the fact that the associated form~$a-p(\eta)$ involves no first-order derivatives.

For the regularity far from~$0$, we start by remarking that the regularity of~$P$ near any point~$(0,z)\in\R^{n-2}\times\Gamma_*$ with~$\la z\ra=1$ comes immediately from the constancy of its coefficients and the classical interior and boundary estimates. The constant can be taken uniformly for those points because they form a compact set. Then applying the associated regularity theory and estimates, to a cut-off near this set times~$e^{i2^\gamma\langle\eta,y\rangle}u(2^\gamma z)$ for~$\gamma\in\N^*$, gives the regularity of~$u$ in dyadic crowns exhausting~$B(0,1)^c$, uniformly in~$2^\gamma$. Then summing the pieces gives the regularity far from~$0$. The constant are of course smoothly dependent on the coefficients of~$p$ and on~$\eta\in S^{n-2}$, which inhabits compact set. Therefore it can be taken uniform in~$\Lambda_*$ and in~$\eta\in S^{n-2}$. 

We derive that for any~$\eta\in\R^{n-2}$, if~$u$ has compact support in~$B(0,1)$, then~$P(\eta)^{(0)}u=I_{0,\sigma}(f,g,h)$ means that~$u\in H^{\sigma+1}$ with estimations uniform in the norms~$H^{r}(\Gamma_*,\la\eta\ra)$. This is immediate for~$\eta\leq 1$ because those norms are equivalent to the classical Sobolev norm, and because~$P(\eta)-P(\eta/\la\eta\ra)$ is compact. 

For~$\eta>1$, we simply apply the regularity of~$P(\eta/\la\eta\ra)$ to~$u(z/\la\eta\ra)$. The fact that the support of this function goes to infinity is the reason we needed estimates far from~$0$ in the preceding step.

Using Fourier, those estimates immediately imply Proposition~\ref{prop:frozenreg}.

\end{proof}

Then the remainder of the proof of Theorem~\ref{thm:ellreg} is simple. We first deduce from the regularity of~$P$ near~$0$ the regularity of~$A$ near~$x\in\c L$, by mapping a neighborhood of~$x$ to a neighborhood of~$0$ in~$\R^{n-2}\times\Gamma_*$ so that~$p$ is its associated frozen coefficients form, then using the  $H^{s-\frac12}$ (with~$s>1+\frac n2$) regularity of the coefficients to write that the form~$a$ is close to the form~$p$ for functions with support small enough, close to~$x$. The size of this support depends only on the Lipschitz norm of those coefficients, so that it is bounded from below for~$\c S\in\Lambda_*$. Those steps are where the regularity of the solution gets limited to~$s+\frac12$.

Then combining this with the classic regularity near other points, and using the compactness of~$\Omega_*$, we can finish the proof. For the last step of regularity, assuming~$\c S$ is~$H^s$, one needs simply to use the~$H^s(\Omega)$ regularity of the solution, with constants uniform in~$\Lambda_*$, and then prove the regularity up to~$H^{s+\frac12}(\Omega)$ by repeating the proof above, noticing that since the analysis is performed on the constant-coefficients problems, the only time we need use their full regularity is in the last step, when deducing the regularity of the full problem from the one with frozen coefficients. It is readily seen that this gives a constant linear in~$\la\c S\ra_s$.

It can be remarked in the proof above that the regularity is limited by the angle only at~$\c L$. Therefore, if~$s\geq\frac{\pi}{n+1}$, the regularity of the solution in the full domain will only be~$H^{1+\frac{pi}{2\omega_{\mathrm{max}}}-}$, but we still have the following.
\begin{lem}
 Take~$s>1+\frac n2$, an~$H^s$ surface~$\c S_*$ as above, and~$\delta>0$ small enough. Assume~$\c S\in\Lambda_*$ is also in~$H^s$. Take~$u$ a variational solution as above, with data~$(f,g,h)\in H^{r+\frac12}\times H^{r}\times H^{r-\frac12}$, with~$0\leq r\leq s-\frac12$. Then~$u$ is locally in~$H^{r+1}$ near any point of~$\bar{\Omega}\setminus\c L$.
\end{lem}

\paragraph{Other elliptic problems.}
Another elliptic problem in~$\Omega$ we need to solve is
\begin{equation*} 	
	\lB
	\begin{aligned}
		&\Delta u=g\text{ in }\Omega,\\
		&\nab Nu\rvert_{\c S}=f,\\
		&\nab\nu u\rvert_{\c B}=h,
	\end{aligned}
	\right.
\end{equation*}
subject to the natural condition
\begin{equation*}
	\int_\Omega g\d x=\int_{\c S}f\d S+\int_{\c B}h\d S.
\end{equation*}
The resolution is completely analogous to the previous one, although it has to be performed modulo constants since they are always solution of the homogeneous problem. The only real change is the model operator, which will of course be the Laplace operator on a sector of angle~$\omega$, but with Neumann condition at both sides. It is readily checked that the singularities appear for~$\lambda=\frac{k\pi}{\omega}$, $k\in\N^*$, and therefore the solution is (modulo constants) in~$H^{r+1}$ with~$0<r\leq \min\lp s-\frac12,\frac{\pi}{\omega_{\mathrm{max}}}\rp$. 

In the same vein, one obtain the same regularity result for the Dirichlet-Dirichlet problem, assuming the two pieces of data can be pasted together at the angle to form a smooth enough function. The singularities are at the same place as in the Neumann-Neumann case.

\subsection{The Dirichlet to Neumann operator}	\label{subsec:DN}
The fundamental operator for the analysis is the Dirichlet to Neumann operator~$\c N$, defined on functions on~$\c S$ by
\begin{equation*}
 \c Nf=\nab N\c H(f),
\end{equation*}
where~$\c H(f)$ is the harmonic extension of~$f$ in~$\Omega$, satisfying
\begin{equation*} 	
	\lB
	\begin{aligned}
		&\Delta \c H(f)=0\text{ in }\Omega,\\
		&\c H(f)\rvert_{\c S}=f,\\
		&\nab\nu\c H(f)\rvert_{\c B}=0.
	\end{aligned}
	\right.
\end{equation*}

As a consequence of the preceding analysis, this operator is continuous, elliptic, and self-adjoint on~$L^2(\c S)$, for~$\c S$ in~$\Lambda_*$. 
\begin{prop}
 Let~$s>1+\frac n2$.
 Let~$\c S_*$ be an~$H^{s-\frac12}$ reference hypersurface, and~$\delta$ small enough, so that in the corresponding~$\Lambda_*$, the maximal angle satisfies~$s<\frac\pi{2\overline\omega}$, and all other geometric conditions imposed above apply. Let~$1\leq\sigma\leq s-\frac12$.
 \begin{enumerate}
  \item Continuity: there is a constant~$C$, depending only on~$\Lambda_*$ and~$\sigma$, such that if~$\c S$ is an hypersurface in~$\Lambda_*$, and~$f\in H^\sigma(\c S)$,then~$\c Nf\in H^{\sigma-1}(\c S)$, and 
  \[\lA\c Nf\rA_{H^{\sigma-1}(\c S)}\leq C\lA f\rA_{H^\sigma(\c S)}.\]
  \item Ellipticity: there is a constant~$C$, depending only on~$\Lambda_*$ and~$\sigma$, such that if~$\c S$ is an hypersurface in~$\Lambda_*$, and~$f\in H^1(\c S)$ is such that~$\c Nf\in H^{\sigma-1}(\c S)$, then~$f\in H^{\sigma}(\c S)$, and 
  \[\lA f\rA_{H^{\sigma}(\c S)}\leq C\lb\lA\c N f\rA_{H^{\sigma-1}(\c S)}+\lA f\rA_{L^2(\c S)}\rb.\]
  \item Self-adjointness on~$L^2(\c S)$: if~$f$ and~$g$ are in~$H^1(\c S)$, then
  \[\int_{\c S}f\c Ng\d S=\int_{\c S}g\c Nf\d S.\]
 \end{enumerate}
 \end{prop}
\begin{proof}
 Continuity and ellipticity are simple consequences of the preceding analysis, for the Dirichlet-Neumann and Neumann-Neumann problem respectively,
 the angle condition being verified for both. To check that~$\c N$ is symmetric, we observe that Stokes formula is valid since the domain is Lipschitz, and then
 \[\int_{\c S}f\c Ng\d S=\int_\Omega\nabla\c Hf\cdot\nabla\c Hg\d x=\int_{\c S}g\c Nf\d S.\] 
 Self-adjointness immediately follows.
\end{proof}

Then by the spectral theorem, one can define non-integer powers of~$\c N$, with the obvious mapping and ellipticity properties.

\subsection{Div-curl Problem}

In order to recover the velocity in the following analysis, we will need an elliptic regularity statement for the following problem on a vector field~$v:\Omega\rightarrow\R^d$.

\begin{equation} 	\label{eq:divcurl}
	\lB
	\begin{aligned}
		&\nabla\cdot v=g\text{ in }\Omega,\\
		&\nabla\times v=\mu\text{ in }\Omega,\\
		&(\langle\nabla v,N\rangle^\sharp)^\top=f \text{ on }\c S\\
		&\langle v,\nu\rangle=h\text{ on }\c B.
	\end{aligned}
	\right.
\end{equation}
Here~$\nabla\times v$ is a shorthand for the vorticity form~$\mu(X)\cdot Y=\langle\nab Xv,Y\rangle-\langle\nab Yv,X\rangle$ which, as is well-known, can be seen as a function in dimension~$n=2$ and a vector field in dimension~$n=3$. In our problem, $g=0$ and we could assume~$\mu=0$, however this does not simplify the proof, and therefore we may as well study the general case.

In this paper, we will not need the existence of a variational solution for such a problem, so we concentrate on regularity theorems. We will have two main difficulties.

The first one is already present in the usual case without corners. The surface~$\c S$ is~$H^s$, and we will want the velocity field~$v\in H^s(\Omega)$. However, if we see~$v$ as a vector field on the $H^{s+\frac12}$ manifold with corner~$\Omega$, its maximal regularity would be~$H^{s-\frac12}(\Omega)$, since change of coordinates on vector fields involve multiplication by~$D\phi_\Omega$. The way out of this is to consider~$v$ as an array of function, whose regularity is only limited at~$H^{s+\frac12}(\Omega)$, and satisfying some additional relations. Therefore any technique invariant by change of coordinates on vector fields, like the Hodge decomposition, is useless in this setting. Remark however that there is no problem of definition for the data, which are all less regular functions.

The second difficulty is again due to the presence of the corner. Since we do not want to perform an analysis of the singularities as above, our aim is to reduce the problem to a scalar one, at least locally near the corner, and use the preceding results.

The reason why the boundary data on~$\c S$ is not under the more classical form~$\ls v,N\rs$ is that~$N$ would limit the regularity of this expression, being only in~$H^{s-\frac32}(\c S)$ for~$\c S\in\Lambda_*$. Since~$\nu$ has maximal regularity on~$\c B$, this does not happen for the other boundary condition.

For the following Proposition, remark that since we only want~$v\in H^s(\c S)$ at the maximum, we do not need to know that~$\c S$ is~$H^s$. The information that it is in~$H^{s-\frac12}$, which is included in the hypothesis~$\c S\in\Lambda_*$, is sufficient. 
\begin{prop}
 Take~$s>1+n/2$, an $H^s$ hypersurface~$\c S_*$ and~$\delta>0$ small enough, so that in particular as above, $s<\frac{\pi}{2\omega_{\mathrm{max}}}+\frac12$. Then if~$\c S\in\Lambda_*$, if~$(g,\mu,f,h)\in H^{\sigma-1}(\Omega)\times H^{\sigma-1}(\Omega)\times H^{\sigma-\frac 32}(\c S)\times H^{\sigma-\frac12}(\c B)$, with~$s-1\leq\sigma\leq s$, if~$v$ is a solution of~\eqref{eq:divcurl} in~$H^\sigma(\Omega)$, then
 \[\lA v\rA_{H^\sigma(\Omega)}\leq C\lb\lA g\rA_{H^{\sigma-1}(\Omega)}+\lA\omega\rA_{H^{\sigma-1}(\Omega)}+\lA f\rA_{H^{\sigma-\frac 32}(\c S)}+\lA h\rA_{H^{\sigma-\frac 12}(\c B)}+\lA v\rA_{L^2(\Omega)}\rb,\]
 where~$C$ is uniform in~$\Lambda_*$.
\end{prop}
The dependence on~$\lA v\rA_{L^2}$ is good, since as a part of the Hamiltonian, it is a bounded quantity. Also the condition on the angle correspond to what would be expected for~$v=\nabla\psi$ and a smooth geometry, if we wanted~$\psi$ to be in~$H^{s+1}$. 

\begin{proof}
First, remark that by interpolation it is enough to prove the inequality with~$\lA v\rA_{H^{s-\epsilon}(\Omega)}$ for some~$\epsilon>0$ in place of the~$L^2$ norm.
 The proof is based on the observation that~\eqref{eq:divcurl} implies that the euclidean coordinates of~$v$ satisfy
 \begin{equation}	\label{eq:lapv}
  \Delta v^i=\partial_j\mu_j^i+\partial_if\text{ in }\Omega,
 \end{equation}
 so that one can study~$v^i$ as a function satisfying an elliptic equation.
 
 The regularity will again be proved locally near any point of~$\overline{\Omega}$. Near interior points, equation~\eqref{eq:lapv} is enough.
 
 Near a boundary point~$x_0$ not in the edge~$\c L$, the analysis is more involved. Since it is no different in both components, we concentrate on~$\c S$. First, we freeze coefficients. More precisely, we use a local coordinate map~$\psi$, of class~$H^s$, and such that~$D\psi$ is the identity at~$x_0$, and a cutoff to transfer the functions~$v^i$ close to~$x$ to compactly supported functions with value~$v^i(\psi)$ close to~$x_0$. Thus, by pulling back the~$v^i$ as functions, we avoid the loss of regularity. Then those functions, close to~$x_0$, satisfy a certain div-curl problem with non-constant coefficients, depending smoothly on~$D\psi^{-1}$, which, if frozen at~$x_0$, give the euclidean divergence and curl operators, and the straight boundary condition. 
 For example on the divergence part, if~$\chi$ denotes the inverse of~$\psi$ and~$v^i=w^i(\chi)$, we have
 \[\partial_i\chi^j\partial_jw^i=g(\psi),\]
 with~$D\chi(0)=I$, and therefore writing
 \[\partial_iw^i=\sum_i\partial_iw^i(1-\partial_i\chi^i)-\sum_{i\neq j}\partial_j\chi^i\partial_iw^j+g\]
 gives the control
 \[\lA\nabla\cdot w\rA_{H^{s-1}}\leq\lA g\rA_{H^{s-1}}+\lA Dw\rA_{H^{s-1}}\lA D\chi-Id\rA_{L^\infty}+C(\chi)\lA Dw\rA_{L^\infty}.\]
 Then since~$D\chi$ is Lipschitz we find that on a ball of radius~$\rho$,
 \[\lA\nabla\cdot w\rA_{H^{s-1}}\leq C\lb\lA g\rA_{H^{s-1}}+\rho\lA w\rA_{H^s}+\lA w\rA_{H^{s-\epsilon}}\rb\]
 for~$\epsilon>0$ small enough, with~$C$ uniform in~$\Lambda_*$, independent of~$\rho$. The same can be done for the curl and for the boundary data. Then one only needs to prove the regularity near~$x_0$ for the straight coefficients, and then take~$\rho$ small enough to absorb the term in the left-hand side. Notice that even if the right-hand side of the original problem was~$0$, we would still need to study the inhomogeneous version since the freezing process produces a right-hand side.
 
 Then one need only to prove the regularity for a solution of~\eqref{eq:divcurl} in a half-space for some function compactly supported near~$0$. We again transform the problem into~\eqref{eq:lapv}, and study each coordinates separately. If~$e_n$ is the normal coordinate in our half-space, the coordinates~$v^i$, for~$i\neq n$ satisfy a Laplace problem, and in terms of boundary data, we find
 \[\nab nv^i=\nab n\ls v,e_i\rs=\ls\nab nv,e_i\rs=\ls\nab{e_i}v,n\rs+\ls(\nabla\times v)(n),e_i\rs,\]
 and both of those terms are part of the data. Therefore we have a Neumann boundary condition for~$v_i$, and the regularity is classical.
 On the other hand, $v^n$ also satisfies a Laplace problem, and 
 \[\nab{e_i}v^n=\ls\nab{e_i}v,n\rs,\]
 which is part of the data for all~$i\neq n$, so that we control the full gradient of~$v^n$ on the boundary, and thus using~$\lA v\rA_{H^{s-\epsilon}}$ we can control the value of~$v^n$ on the boundary. We therefore have a Dirichlet problem for~$v^n$, and we can easily conclude.
 
 The last step is the control near a point of~$\c L$. As above, we reduce it to the same problem in the angular sector~$\R^{n-2}\times\Gamma$. The components of~$v$ in the unbounded direction, being tangential to all parts of the boundary, can be treated as above. We are left with two components, and we would like to reduce the problem to a 2D system in the angular sector~$\Gamma$. As in the proof of the scalar problem, we apply the Fourier transform in the unbounded variables. Treating all terms in the tangential variables and all lower order terms as a right-hand side, we have a problem with parameter~$\xi\in\R^{n-2}$, of the form
 \begin{equation}
	\lB
	\begin{aligned}
		&\nabla\cdot \tilde{v}(\xi)=\tilde{g}(\xi)\text{ in }\Gamma,\\
		&\nabla\times \tilde{v}=\tilde{\mu}(\xi)\text{ in }\Gamma,\\
		&(\langle\nabla\tilde{v}(\xi),N\rangle^\sharp)^\top=\tilde{f}(\xi) \text{ on } S\\
		&(\langle\nabla\tilde{v}(\xi),\nu\rangle^\sharp)^{\top_b}=\tilde{h}(\xi)\text{ on } B.
	\end{aligned}
	\right.
\end{equation}
 Here~$\tilde{v}(\xi)$ is for each~$\xi$ a vector field on~$\Gamma$. We need to prove weighted in~$\xi$ estimates for this problem. Again, as in the scalar case, it is sufficient to prove them for fixed~$\xi$ on the sphere, but with arbitrary large support, and constants independents of the support.
 
 Now since we are on a straight domain, there is no problem anymore to use vectorial methods. One can use Hodge decomposition in the angular sector, which exists since~$\Gamma$ is piecewise regular and the function has compact support, writing~$\tilde{v}=\nabla\phi+\nabla^\perp\tilde\phi$, where~$\nabla^\perp$ here is the perpendicular gradient~$(\partial_2,-\partial_1)$, and such that~$\nab N\phi=\tilde{v}\cdot N$ on~$S$ and the same on~$B$. Differentiating those boundary condition along the tangential direction~$\tau$, we find that we control
 \[\nab\tau\nab N\phi=\ls\nab\tau\tilde v,N\rs=\tilde f\]
 and since we also control~$v$ in~$H^{s-\epsilon}$, we control the Neumann data~$\partial_N\phi$. The same can be done on the bottom~$B$, so that we can use~$\Delta\phi=\dive\tilde v=\tilde g$ and our preceding regularity result for functions satisfying some elliptic problem with Neumann data on both sides to control~$\nabla\phi$ as expected. 
 
 At last, we find~$\Delta\tilde\phi=\nabla\times\tilde v=\tilde g$, and since~$\ls\nabla^\perp\tilde\phi,N\rs=\nab\tau\tilde\phi=0$, we find~$\tilde\phi$ to be constant on the boundary, and we can again use elliptic regularity for functions  with Dirichlet data at both sides to conclude.
 
\end{proof}

We notice that for the case~$\sigma=s-1$, the boundary data should be given under the form~$\ls v,N\rs$ and~$\ls v,\nu\rs$, because~$N\in H^{s-\frac32}(\c S)$ and therefore does not limit the regularity.

\section{Quasi-linearization}	\label{sec:quasilin}
In this section, we want to find a quantity satisfying the linearized equation. As explained in the introduction, we want to differentiate the equation in time. Since
\[\D_tv=-\nabla p-ge_n,\]
we look for an equation on~$\nabla p=\nabla p_{v,v}-g\nabla\c H(x^n\rvert_{\c S})+g\nabla x^n$ (the term~$ge_n$ does not depend on time.) As we will see later, the regularity of~$\nabla p$ is equivalent to the regularity of its normal part~$a=-\nab Np\rvert_{\c S}$, the Taylor coefficient, so we will in fact prove that~$a$ satisfies the linearized equation. 

\begin{prop}	\label{prop:quasilin}
 Let~$s>1+\frac n2$.
 Let~$\c S_*$ be an~$H^s$ reference hypersurface, and~$\delta$ small enough, so that in the corresponding~$\Lambda_*$, the maximal angle satisfies~$s<\frac12+\frac\pi{2\overline\omega}$, and all other geometric conditions imposed above apply. 
 
 Then if~$\c S_t$ is a continuous family of~$H^s$ hypersurface belonging to~$\Lambda_*$, if~$v\in C(H^s(\Omega_t))$, and if they satisfy the water waves equations, then 
 the Taylor coefficient~$a\in C(H^{s-1}(\c S_t))$ follows the equation
 \begin{equation}
  \D_t^2a+a\c Na=R.
 \end{equation}
 Here, the remainder~$R$ is in~$C(H^{s-\frac32}(\c S_t))$, with at each time~$t$
 \begin{equation}
  \lA R(t)\rA_{H^{s-\frac32}(\c S_t)}\leq Q\lp\la\c S_t\ra_s,\lA v\rA_{H^s(\Omega_t)}\rp,
 \end{equation}
 where~$Q$ is a time-independent polynomial in its variables, whose coefficients depend only on~$\Lambda_*$ and~$g$.
\end{prop}

\begin{proof}
 This is only a long computation. To start with, we recall that~$p$ is the solution of the elliptic equation
 \begin{equation} 	\label{eq:pressell}
	\lB
	\begin{aligned}
		&\Delta p=-\mathrm{tr}\lp(Dv)^2\rp\text{ in }\Omega_t,\\
		&p\rvert_{\c S_t}=0,\\
		&\nab\nu p\rvert_{\c B_t}=-\Pi_{\c M}(v,v)+g\nu^n
	\end{aligned}
	\right.
\end{equation}

where~$\Pi_{\c M}(v,w)=-\ls\nab{\nu}v,w\rs$ is the second fundamental form of~$\c M$ and~$\nu^n$ is the component of~$\nu$ along~$e_n$. 
As a consequence of Theorem~\ref{thm:ellreg}, recalling that~$\Pi_{\c M}$ is smooth, we find that~$p\in C(H^{s+\frac12}(\Omega_t))$ with for each time
\begin{equation}
 \lA p\rA_{H^{s+\frac12}(\Omega_t)}\leq Q\lp\la\c S_t\ra_s,\lA v\rA_{H^{s-\frac12}(\Omega_t)}\rp
\end{equation}
where~$Q$ is as in the Proposition. The fact that we can use only the~$H^{s-\frac12}$ norm of~$v$ means that the regularity of~$p$ is limited by the domain, and not its data.

Then
\begin{equation}	\label{eq:evgradpprep}
 \D_t\nabla p=-\ls\nabla v,\nabla p\rs^\sharp+\nabla\D_tp,
\end{equation}
where~$\sharp$ is the musical isomorphism corresponding to raising indices with the metric.
We want to find the elliptic problem satisfied by~$\D_tp$.

\paragraph{Elliptic problem for~$\D_tp$.}
First, using~\eqref{eq:pressell},
\begin{equation}
 \D_tp\rvert_{\c S_t}=\D_t\lp p\rvert_{\c S_t}\rp=0.
\end{equation}

In~$\Omega_t$,
\begin{equation*}
 \Delta\D_tp=\D_t\Delta p+2\tr\lp D^2p\cdot Dv\rp+\ls\Delta v,\nabla p\rs,
\end{equation*}
where~$\cdot$ represents the matrix product. But, using~\eqref{eq:pressell}, we find
\begin{equation*}
 \D_t\Delta p=-\D_t\tr\lp(Dv)^2\rp=-2\tr\lp D\D_t v\cdot Dv\rp+2\tr\lp(Dv)^3\rp,
\end{equation*}
and the Euler equations give
\begin{equation*}
 \D_t\Delta p=2\tr\lp D^2p\cdot Dv\rp+2\tr\lp(Dv)^3\rp.
\end{equation*}
Thus we have proved that in~$\Omega_t$,
\begin{equation}
 \Delta\D_tp=4\tr\lp D^2p\cdot Dv\rp+2\tr\lp(Dv)^3\rp+\ls\Delta v,\nabla p\rs.
\end{equation}

Since both~$\nu$ and~$\Pi_{\c M}$ are smooth and independent of~$t$, we compute easily that on points of~$\c B_t$,
\begin{equation}	\label{eq:evnorbot}
 \D_t\nu\rvert_{\c B}=\nab v\nu
\end{equation}
and for~$w$ and~$w'$ tangent to~$\c B$,
\begin{equation}	\label{eq:ev2fundbot}
 \D_t\lp\Pi_{\c M}(w,w')\rp=\Pi_{\c M}(\D_t^\top w,w')+\Pi_{\c M}(w,\D_t^\top w')-\ls D^2\nu(v,w),w'\rs,
\end{equation}
where~$\top$ represents the orthogonal projection to the tangent plane of~$\c M$.

On the other hand,
\begin{equation*}
 \nab\nu\lp\D_tp\rp\rvert_{\c B_t}=\D_t\lp\nab\nu p\rvert_{\c B_t}\rp+\ls\nab\nu v,\nabla p\rs\rvert_{\c B_t}+\Pi_{\c M}(v,\nabla^\top p). 
\end{equation*}
Using the elliptic equation~\eqref{eq:pressell} on~$p$ and the preceding computations~\eqref{eq:evnorbot} and~\eqref{eq:ev2fundbot}, we find
\begin{equation*}
 \D_t\lp\nab\nu p\rvert_{\c B_t}\rp=-2\Pi_{\c M}(\D_t^\top v,v)+\ls D^2\nu(v,v),v\rs\rvert_{\c B_t}+g\lp\nab v\nu\rp^n\rvert_{\c B_t}.
\end{equation*}
Euler equations let us conclude
\begin{equation*}
 \D_t\lp\nab\nu p\rvert_{\c B_t}\rp=2\Pi_{\c M}(\nabla^\top p,v)+\ls D^2\nu(v,v),v\rs\rvert_{\c B_t}-g\lp\nab v\nu\rp^n\rvert_{\c B_t}.
\end{equation*}
Therefore,
\begin{equation}
  \nab\nu\lp\D_tp\rp\rvert_{\c B_t}=3\Pi_{\c M}(\nabla^\top p,v)+\ls D^2\nu(v,v),v\rs\rvert_{\c B_t}-g\lp\nab v\nu\rp^n\rvert_{\c B_t}+\ls\nab\nu v,\nabla p\rs\rvert_{\c B_t}. 
\end{equation}

By grouping together terms of same regularity, we thus find for~$\D_tp$ the expression
\begin{equation}	\label{eq:ev1p}
\begin{aligned}
 \D_tp=&\Delta^{-1}\lp\ls\Delta v,\nabla p\rs,\ls\nab\nu v,\nabla p\rs\rp\\
 &+\Delta^{-1}\lp 4\tr\lp D^2p\cdot Dv\rp+2\tr\lp(Dv)^3\rp,3\Pi_{\c M}(\nabla^\top p,v)+\ls D^2\nu(v,v),v\rs-g(\nab v\nu)^n\rp,
 \end{aligned}
\end{equation}
where~$\Delta^{-1}(g,h)$ is the solution of
\begin{equation*} 	
	\lB
	\begin{aligned}
		&\Delta u=g\text{ in }\Omega_t,\\
		&u\rvert_{\c S_t}=0,\\
		&\nab\nu u\rvert_{\c B_t}=h.
	\end{aligned}
	\right.
\end{equation*}

The elliptic regularity Theorem~\ref{thm:ellreg}, combined with product estimates gives us that the first term of~\eqref{eq:ev1p} is in~$H^s(\Omega_t)$, while the second is in~$H^{s+\frac12}$.

If we plug this into~\eqref{eq:evgradpprep}, we find
\begin{equation}	\label{eq:ev1gradp}
\begin{aligned}
 \D_t\nabla p=&-\ls\nabla v,\nabla p\rs^\sharp+\nabla\Delta^{-1}\lp\ls\Delta v,\nabla p\rs,\ls\nab\nu v,\nabla p\rs\rp\\
 &+\nabla\Delta^{-1}\lp 4\tr\lp D^2p\cdot Dv\rp+2\tr\lp(Dv)^3\rp,3\Pi_{\c M}(\nabla^\top p,v)+\ls D^2\nu(v,v),v\rs-g(\nab v\nu)^n\rp.
 \end{aligned}
\end{equation}
The first line of the right-hand side is in~$H^{s-1}(\Omega_t)$, the second in~$H^{s-\frac12}(\Omega_t)$.
Let us call 
\begin{equation*}
 \alpha:=\Delta^{-1}\lp\ls\Delta v,\nabla p\rs,\ls\nab\nu v,\nabla p\rs\rp,
\end{equation*}
and
\begin{equation*}
 \beta:=\Delta^{-1}\lp 4\tr\lp D^2p\cdot Dv\rp+2\tr\lp(Dv)^3\rp,3\Pi_{\c M}(\nabla^\top p,v)+\ls D^2\nu(v,v),v\rs-g(\nab v\nu)^n\rp,
\end{equation*}
so that in short,
\begin{equation}	\label{eq:ev1gradpshort}
 \D_t\nabla p=-\ls\nabla v,\nabla p\rs^\sharp+\nabla\alpha+\nabla\beta.
\end{equation}

Now we need to compute the second derivative in time. For this, we compute the derivative of each of the three terms of~\eqref{eq:ev1gradpshort}.

\paragraph{First term.}
We first compute
\begin{equation*}
 -\D_t\ls\nabla v,\nabla p\rs^\sharp=-\ls\nabla\D_tv,\nabla p\rs^\sharp+\ls\nab{\nabla v}v,\nabla p\rs^\sharp-\ls\nabla v,\D_t\nabla p\rs^\sharp.
\end{equation*}
From Euler equations, the first term is
\begin{equation*}
 -\ls\nabla\D_tv,\nabla p\rs^\sharp=\ls D^2p,\nabla p\rs^\sharp.
\end{equation*}
Using the evolution of~$\nabla p$ from~\eqref{eq:ev1gradpshort} to express the third term, we find
\begin{equation*}
 -\D_t\ls\nabla v,\nabla p\rs^\sharp=\ls D^2p,\nabla p\rs^\sharp+2\ls\nab{\nabla v}v,\nabla p\rs^\sharp-\ls\nabla v,\nabla\alpha\rs^\sharp-\ls\nabla v,\nabla\beta\rs^\sharp.
\end{equation*}
Using the product estimates we can sum this up as
\begin{equation}
 -\D_t\ls\nabla v,\nabla p\rs^\sharp=\ls D^2p,\nabla p\rs^\sharp+R
\end{equation}
with~$R$ satisfying
\begin{equation}
 \lA R\rA_{H^{s-1}(\Omega_t)}\leq Q\lp\la\c S_t\ra_s,\lA v\rA_{H^s(\Omega_t)}\rp
\end{equation}
with~$Q$ as in the Proposition.

\paragraph{Second term.}
We now compute
\begin{equation}
 \D_t\nabla\alpha=\nabla\D_t\alpha-\ls\nabla v,\nabla\alpha\rs^\sharp.
\end{equation}
To find an expression for~$\D_t\alpha$, we use the same method as for~$p$. We look for an elliptic problem it satisfies.

It is immediate that
\begin{equation}	\label{eq:evalphaS}
 \D_t\alpha\rvert_{\c S_t}=0.
\end{equation}

In the domain~$\Omega_t$,
\begin{equation}	\label{eq:lapevalpha1}
 \Delta\D_t\alpha=\D_t\Delta\alpha+2\tr\lp D^2\alpha\cdot Dv\rp+\ls\Delta v,\nabla\alpha\rs.
\end{equation}
Let us concentrate on the first term, using that
\begin{equation*}
 \Delta\alpha=\ls\Delta v,\nabla p\rs.
\end{equation*}
We find
\begin{equation*}
 \D_t\Delta\alpha=\ls\Delta\D_tv,\nabla p\rs-2\tr\lp\ls D^2v,\nabla p\rs\cdot Dv\rp-\ls\nab{\Delta v}v,\nabla p\rs+\ls\Delta v,\D_t\nabla p\rs.
\end{equation*}
Thanks once again to Euler's equations, 
\begin{equation*}
 \ls\Delta\D_tv,\nabla p\rs=-\ls\Delta\nabla p,\nabla p\rs=\ls\nabla\tr\lp(Dv)^2\rp,\nabla p\rs,
\end{equation*}
so that again, using the evolution of~$\nabla p$ from~\eqref{eq:ev1gradpshort}, 
\begin{equation*}
 \D_t\Delta\alpha=\ls\nabla\tr\lp(Dv)^2\rp,\nabla p\rs-2\tr\lp\ls D^2v,\nabla p\rs\cdot Dv\rp-2\ls\nab{\Delta v}v,\nabla p\rs+\ls\Delta v,\nabla\alpha\rs+\ls\Delta v,\nabla\beta\rs.
\end{equation*}
Combined with~\eqref{eq:lapevalpha1}, this gives
\begin{multline}	\label{eq:evalphaDel}
 \D_t\Delta\alpha=\ls\nabla\tr\lp(Dv)^2\rp,\nabla p\rs-2\tr\lp\ls D^2v,\nabla p\rs\cdot Dv\rp\\
 -2\ls\nab{\Delta v}v,\nabla p\rs+2\ls\Delta v,\nabla\alpha\rs+\ls\Delta v,\nabla\beta\rs+2\tr\lp D^2\alpha\cdot Dv\rp.
\end{multline}

At last on~$\c B_t$,
\begin{equation*}
 \nab\nu\D_t\alpha=\D_t(\nab\nu\alpha)+\ls\nab\nu v,\nabla\alpha\rs-\ls\nab v\nu,\nabla \alpha\rs.
\end{equation*}
The first term can be computed from
\begin{equation*}
 \nab\nu\alpha\rvert_{\c B_t}=\ls\nab\nu v,\nabla p\rs,
\end{equation*}
and is
\begin{equation*}
 \D_t(\nab\nu\alpha)=\ls\nab\nu\D_tv,\nabla p\rs+\ls\nab{\nab v\nu}v,\nabla p\rs-\ls\nab{\nab\nu v}v,\nabla p\rs+\ls\nab\nu v,\D_t\nabla p\rs.
\end{equation*}
Using Euler equations and~\eqref{eq:ev1gradpshort} once again,
\begin{equation}	\label{eq:evalphaB}
 \nab\nu\D_t\alpha=-D^2p(\nu,\nabla p)+\ls\nab{\nab v\nu}v,\nabla p\rs-2\ls\nab{\nab\nu v}v,\nabla p\rs+2\ls\nab\nu v,\nabla\alpha\rs+\ls\nab\nu v,\nabla\beta\rs+\Pi_{\c M}(v,\nabla^\top\alpha).
\end{equation}

Combining~\eqref{eq:evalphaS}, \eqref{eq:evalphaDel} and~\eqref{eq:evalphaB} with elliptic and product estimates, we find
\begin{equation}
 \D_t\nabla\alpha=-\nabla\Delta^{-1}\lp0,D^2p(\nu,\nabla p)\rp +R
\end{equation}
with
\begin{equation}
 \lA R\rA_{H^{s-1}(\Omega_t)}\leq Q\lp\la\c S_t\ra_s,\lA v\rA_{H^s(\Omega_t)}\rp.
\end{equation}

\paragraph{Third term.}
We finish by computing
\begin{equation}
 \D_t\nabla\beta=\nabla\D_t\beta-\ls\nabla v,\nabla\beta\rs^\sharp.
\end{equation}
Again we look for an elliptic problem on~$\D_t\beta$.

As above,
\begin{equation}
 \D_t\beta\rvert_{\c S_t}=0.
\end{equation}

In~$\Omega_t$, 
\begin{equation*}
 \Delta\D_t\beta=\D_t\Delta\beta+2\tr\lp D^2\beta\cdot Dv\rp+\ls\Delta v,\beta\alpha\rs,
\end{equation*}
and
\begin{multline*}
 \D_t\Delta\beta=4\tr\lp D\D_t\nabla p\cdot Dv\rp+4\tr\lp D^2p\cdot D\D_t v\rp-8\tr\lp D^2p\cdot Dv\cdot Dv\rp\\
 +6\tr\lp D\D_tv\cdot Dv\cdot Dv\rp-6\tr\lp(Dv)^4\rp.
\end{multline*}
At the end end, each of those terms lies in~$H^{s-2}(\Omega_t)$, so that
\begin{equation}
 \lA \Delta\D_t\beta\rA_{H^{s-2}(\Omega_t)}\leq Q\lp\la\c S_t\ra_s,\lA v\rA_{H^s(\Omega_t)}\rp.
\end{equation}

On the bottom,
\begin{equation*}
 \nab\nu\D_t\beta\rvert_{\c B_t}=\D_t\lp\nab\nu\beta\rvert_{\c B_t}\rp+\ls\nab\nu v,\nabla\beta\rs-\ls\nab v\nu,\nabla\beta\rs,
\end{equation*}
Since
\begin{equation*}
 \nab\nu\beta\rvert_{\c B_t}=3\Pi_{\c M}\lp\nabla^\top p,v\rp+\ls D^2\nu(v,v),v\rs-g\lp\nab v\nu\rp^n,
\end{equation*}
we only need to compute the evolution of each of those three terms.
The first gives
\begin{equation*}
 3\D_t\Pi_{\c M}\lp\nabla^\top p,v\rp=3\Pi_{\c M}\lp\D_t^\top\nabla^\top p,v\rp+3\Pi_{\c M}\lp\nabla^\top p,\D_t^\top v\rp-3\ls D^2\nu(\nabla^\top p,v),v\rs,
\end{equation*}
which is easily seen to lie in~$H^{s-\frac32}(\c B_t)$.
For the second term,
\begin{equation*}
 \D_t\ls D^2\nu(v,v),v\rs=\ls D^4\nu(v,v,v),v\rs+2\ls D^2\nu(\D_tv,v),v\rs+\ls D^2\nu(v,v),\D_tv\rs,
\end{equation*}
which is again in~$H^{s-\frac32}(\c B_t)$.
The last term is 
\begin{equation*}
 -g\D_t\lp\nab v\nu\rp^n=-g\lp\nab{\D_tv}\nu\rp^n+g\lp D^2\nu(v,v)\rp^n,
\end{equation*}
still in~$H^{s-\frac32}(\c B_t)$.
Putting all of this together, we see that
\begin{equation}
 \lA \nab\nu\D_t\beta\rA_{H^{s-\frac32}(\c B_t)}\leq Q\lp\la\c S_t\ra_s,\lA v\rA_{H^s(\Omega_t)}\rp.
\end{equation}

Thus in the end, using elliptic regularity,
\begin{equation}
 \lA\D_t\beta\rA_{H^s(\Omega_t)}\leq Q\lp\la\c S_t\ra_s,\lA v\rA_{H^s(\Omega_t)}\rp.
\end{equation}

\paragraph{Equation on~$\nabla p$.}
We have therefore proved that
\begin{equation*}
 \D_t^2\nabla p=\ls D^2p,\nabla p\rs^\sharp-\nabla\Delta^{-1}\lp0,D^2p(\nu,\nabla p)\rp+R
\end{equation*}
where
\begin{equation*}
 \lA R\rA_{H^{s-1}(\Omega_t)}\leq Q\lp\la\c S_t\ra_s,\lA v\rA_{H^s(\Omega_t)}\rp.
\end{equation*}
To finish, we remark that the first two terms of the right-hand side can be rewritten as 
\begin{equation*}
 \frac12\nabla\lp\la\nabla p\ra^2\rp-\frac12\nabla\Delta^{-1}\lp0,\nab\nu\la\nabla p\ra^2\rp=\frac12\nabla\c H\lp\la\nabla p\ra^2\rvert_{\c S_t}\rp-\nabla f,
\end{equation*}
where~$f$ is solution of
 \begin{equation*} 
	\lB
	\begin{aligned}
		&\Delta f=2\ls\nabla p,\nabla\tr\lp(Dv)^2\rp\rs+2\tr\lp(D^2p)^2\rp\text{ in }\Omega_t,\\
		&f\rvert_{\c S_t}=0,\\
		&\nab\nu f\rvert_{\c B_t}=0,
	\end{aligned}
	\right.
\end{equation*}
and is therefore in~$H^s(\Omega_t)$.
Remarking that on~$\c S_t$, because~$p\rvert_{\c S_t}=0$, $\nabla p=\nab NpN$, we find
\begin{equation*}
 \D_t^2\nabla p=\frac12\nabla\c H(a^2)+R
\end{equation*}
with~$R$ as above.
To finish, we notice that
\begin{equation*}
 \c H(a^2)=\lp\c H(a)\rp^2+H^{s+\frac12}(\Omega_t)
\end{equation*}
because of the elliptic problem satisfied by the difference, so that in the end,
\begin{equation}	\label{eq:ev2gradp}
 \D_t^2\nabla p=\c H(a)\nabla\c H(a)+R
\end{equation}
with
\begin{equation}
 \lA R\rA_{H^{s-1}(\Omega_t)}\leq Q\lp\la\c S_t\ra_s,\lA v\rA_{H^s(\Omega_t)}\rp.
\end{equation}

\paragraph{Equation on~$a$.}

To transform this into an equation on~$a$, we need to compute the evolution of the normal to the surface, $\D_tN$. We redo the computations of Shatah and Zheng in~\cite{ShatahZeng1}. First, because~$\la N\ra=1$, 
\[\D_tN\perp N.\]
Then we can choose~$\tau_0$ tangent to~$\c S_{t_0}$ at the point~$x_0\in\c S_{t_0}$, and transport it in time as a solution of
\begin{equation*}
 \D_t\tau=\nab\tau v,\quad\tau(t_0)=\tau_0.
\end{equation*}
Immediatly,
\begin{equation*}
 \ls\D_t N,\tau_0\rs=-\ls N,\D_t\tau\rs=-\ls\nab{\tau_0}v,N\rs.
\end{equation*}
Thus
\begin{equation}	\label{eq:EvN}
 \D_tN=-\lp\ls\nabla v,N\rs^\sharp\rp^\top\in H^{s-\frac32}(\c S_t).
\end{equation}
This also gives
\begin{equation}
 \ls\D_t^2N,N\rs=-\la\D_tN\ra^2\in H^{s-\frac32}(\c S_t).
\end{equation}

Now since~$a=-\ls\nabla p,N\rs\rvert_{\c S_t}$, we find
\begin{equation*}
 \D_t^2a=-\ls\D_t^2\nabla p,N\rs-2\ls\D_t\nabla p,\D_tN\rs+a\ls\D_t^2N,N\rs,
\end{equation*}
and the last two terms are in~$H^{s-\frac32}(\c S_t)$. Thus by taking the scalar product with~$N$ of the trace of~\eqref{eq:ev2gradp}, 
\begin{equation*}
 \D_t^2a+a\c Na=H^{s-\frac32}(\c S_t),
\end{equation*}
with estimates on the remainder as in the Proposition.

\end{proof}

\section{The energy}	\label{sec:Energy}

Using this quasi-linear form for the equations, an obvious choice for the energy is
\begin{equation}
 E:=\int_{\c S_t}\la\c N^{s-\frac32}\D_ta\ra^2\d S+\int_{\c S_t}a\la\c N^{s-1}a\ra^2\d S+\lA\mu\rA^2_{H^{s-1}(\Omega_t)}.
\end{equation}
Here, we use the powers of the Dirichlet to Neumann map~$\c N$, defined in Subsection~\ref{subsec:DN}, and we integrate according to the surface element~$\d S$ of~$\c S_t$. The first two terms correspond to the energy for the linearized equation satisfied by~$a$. We will see that those terms do not control the vorticity part of~$v$, and therefore we added the~$H^{s-1}(\Omega_t)$ norm of~$\mu:=Dv-Dv^*$, which will be well controlled because the vorticity is a conservation law for Euler equations. 

We first need to show that one can recover the original unknowns from this energy.
\begin{prop}	\label{prop:controlE}
 Let~$s>1+\frac n2$.
 Let~$\c S_*$ be an~$H^s$ reference hypersurface, and~$\delta$ small enough, so that in the corresponding~$\Lambda_*$, the maximal angle satisfies~$s<\frac12+\frac\pi{2\overline\omega}$, and all other geometric conditions imposed above apply. 
 
 Then if~$\c S$ is an~$H^s$ hypersurface belonging to~$\Lambda_*$ and~$v\in H^s(\Omega_t)$ are solution of the equations, and if the Taylor condition~$a\geq a_0>0$ is satisfied, then~$E$ is well-defined, finite, and we have
 \[\la\c S\ra^2_s+\lA v\rA^2_{H^s(\Omega)}\leq \c F\lp\lA v\rA_{H^{s-\frac12}(\Omega)},E\rp\]
where~$\c F$ is a non-decreasing continuous function of its arguments, depending only on~$\Lambda_*$, $a_0$ and~$\underline{\omega}$.
\end{prop}

\begin{proof}
 If~$\c S$ is in~$H^s$ and~$v$ is in~$H^s(\Omega)$, our elliptic regularity theory gives~$\nabla p$ in~$H^{s-\frac12}$ and therefore~$a\in H^{s-1}(\c S)$, while the formula for~$\D_ta$ tells us it lies in~$H^{s-\frac32}(\c S)$. Thus the quantities composing~$E$ are all well-defined. 
 
 Now to prove our result, we start with the basic preliminary controls 
 \begin{equation*}
  \lA \nabla p\rA_{H^{s-1}(\Omega)}\leq C\lA v\rA^2_{H^{s-\frac12}(\Omega)}
 \end{equation*}
 because of the elliptic equation~\eqref{eq:pressell} on the pressure, and
 \begin{equation*}
  \lA\D_t\nabla p\rA_{H^{s-\frac32}}(\Omega)\leq\c F\lp\lA v\rA_{H^{s-\frac12}(\Omega)}\rp,
 \end{equation*}
 because of~\eqref{eq:ev1gradpshort} and the elliptic equations on~$\alpha$ and~$\beta$.
 We notice that
 \[\D_ta=-\ls\D_t\nabla p,N\rs-\ls\nabla p,\D_tN\rs=-\ls\D_t\nabla p,N\rs,\]
 because~$\D_t N$ is tangent to~$\c S$ and~$\nabla p$ is normal. Thus
 \begin{equation}
  \lA a\rA_{H^{s-\frac32}(\c S)}+\lA\D_t a\rA_{H^{s-2}(\c S)}\leq\c F\lp\lA v\rA_{H^{s-\frac12}(\Omega)}\rp.
 \end{equation}
 Here the traces on~$\c S$ are well-defined because~$s>2$.
 
 The next step is to use the ellipticity of~$\c N$ to find
 \begin{align*}
  \lA a\rA_{H^{s-1}(\c S)}&\leq \c F\lp\lA v\rA_{H^{s-\frac12}(\Omega)}\rp\lb\lA\c N^{s-1}a\rA_{L^2(\c S)}+1\rb\\
  &\leq\c F\lp\lA v\rA_{H^{s-\frac12}(\Omega)}\rp\lb\lA\frac1{\sqrt a}\rA_{L^\infty(\c S)}\lA\sqrt a\c N^{s-1}a\rA_{L^2(\c S)}+1\rb,
 \end{align*}
 so that
 \begin{equation*}
  \lA a\rA^2_{H^{s-1}(\c S)}\leq \c F\lp\lA v\rA_{H^{s-\frac12}(\Omega)}\rp\lb1+E\rb.
 \end{equation*}
Similarly,
 \begin{equation*}
  \lA\D_ta\rA^2_{H^{s-\frac32}(\c S)}\leq \c F\lp\lA v\rA_{H^{s-\frac12}(\Omega)}\rp\lb1+E\rb.
 \end{equation*}

 Next, we want to control~$\nabla p$ and~$\D_t\nabla p$ with those quantities. To avoid the apparition of~$\la\c S\ra_s$ in the right-hand side, we use the div-curl problems
 \begin{equation*} 	
	\lB
	\begin{aligned}
		&\nabla\cdot \nabla p=-\tr\lp(Dv)^2\rp\text{ in }\Omega,\\
		&\nabla\times \nabla p=0\text{ in }\Omega,\\
		&(\ls\nabla\nabla p,N\rs^\sharp)^\top=-(\nabla a)^\top \text{ on }\c S,\\
		&\ls\nabla p,\nu\rs=-\Pi_{\c M}(v,v)-g\nu^n\text{ on }\c B;
	\end{aligned}
	\right.
\end{equation*}
and
\begin{equation*} 	
	\lB
	\begin{aligned}
		&\nabla\cdot\D_t\nabla p=3\tr\lp D^2p\cdot Dv\rp+2\tr\lp(Dv)^3\rp\text{ in }\Omega,\\
		&\nabla\times\D_t\nabla p=D^2p\cdot Dv-(Dv)^*\cdot D^2p\text{ in }\Omega,\\
		&\ls\D_t\nabla p,N\rs=-\D_ta \text{ on }\c S,\\
		&\ls\D_t\nabla p,\nu\rs=3\Pi_{\c M}(\nabla^\top p,v)+\ls D^2\nu(v,v),v\rs+g\Pi_{\c M}(v,e_n)\text{ on }\c B.
	\end{aligned}
	\right.
\end{equation*}

If~$s$ was greater than~$3/2+n/2$, we could use the information~$v\in H^{s-\frac12}(\Omega)$ directly to conclude. However, here with~$s>1+n/2$, this information only gives~$\tr\lp (Dv)^3\rp\in H^{s-\frac52+}(\Omega)$ for example, and this is not enough to conclude. We therefore need to implement a bootstrap procedure. We see form the above elliptic problems that if~$s-\frac12<\sigma\leq s$, and if we choose~$\epsilon\leq\sigma-s+\frac12$, 
\begin{equation}
 \lA\nabla p\rA_{H^{\sigma-\frac12}(\Omega)}+\lA\D_t\nabla p\rA_{H^{\sigma-1}(\Omega)}\leq\c F\lp\lA v\rA_{H^{\sigma-\eps}(\Omega)},E\rp.
\end{equation}
Here we have first proved the estimate on~$\nabla p$, then used it to prove the one on~$\D_t\nabla p$.

Then we use the problem satisfied by~$v$,
\begin{equation*} 	
	\lB
	\begin{aligned}
		&\nabla\cdot v=0\text{ in }\Omega,\\
		&\nabla\times v=\mu\text{ in }\Omega,\\
		&(\langle\nabla v,N\rangle^\sharp)^\top=\frac1a(\D_t\nabla p)^\top \text{ on }\c S,\\
		&\langle v,\nu\rangle=0\text{ on }\c B,
	\end{aligned}
	\right.
\end{equation*}
to deduce
\begin{equation*}
 \lA v\rA_{H^\sigma(\Omega)}\leq\c F\lp\lA v\rA_{H^{s-\frac12}(\Omega)}\rp\lb1+E^{\frac12}+\lA\D_t\nabla p\rA_{H^{\sigma-\frac12}(\Omega)}\rb.
\end{equation*}
A simple bootstrap procedure closes the estimates: we find
\begin{equation}
 \lA\nabla p\rA_{H^{s-\frac12}(\Omega)}+\lA\D_t\nabla p\rA_{H^{s-1}(\Omega)}+\lA v\rA_{H^s(\Omega)}\leq\c F\lp\lA v\rA_{H^{s-\frac12}(\Omega)},E\rp.
\end{equation}

To control the regularity of~$\c S$, we use the formula
\begin{equation*}
 \Delta p=\Delta_{\c S}p-\kappa\nab Np+D^2p(N,N),
\end{equation*}
so that, since~$p\rvert_{\c S}=0$,
\begin{equation*}
 \kappa=\frac1a\lp\Delta p-D^2p(N,N)\rp.
\end{equation*}
Thus we can conclude that
\begin{equation}
 \lA\kappa\rA_{H^{s-2}(\c S)}\leq\c F\lp\lA v\rA_{H^{s-\frac12}(\Omega)},E\rp.
\end{equation}
If~$n=2$, this is enough. For~$n\geq3$, we need also to control~$\kappa_l$. The same formula as above, seeing~$\c L$ as the boundary of~$\c B$ with exterior normal~$n$ gives 
\begin{equation*}
 \kappa_l=-\frac1{\nab np}\lp\Delta_{\c B}p-D^2p(n,n)\rp=\frac1{\ls n,N\rs a}\lp\Delta_{\c B}p-D^2p(n,n)\rp,
\end{equation*}
where we have used that~$p=0$ on~$\c S$ to write~$\nab np=\ls N,n\rs\nab Np$. Observing that~$\ls n,N\rs$ is bounded from below because~$\frac\pi2>\omega\geq\underline{\omega}>0$, we conclude
\begin{equation}
 \lA\kappa_l\rA_{H^{s-\frac52}(\c L)}\leq\c F\lp\lA v\rA_{H^{s-\frac12}(\Omega)},E\rp.
\end{equation}
Here the traces make sense because~$s>\frac52$ when~$n\geq3$.
This concludes the proof.

\end{proof}

The last proposition is the control on the energy. We need to use a control neighborhood both in~$\c S$, which is the role of~$\Lambda_*$, and in~$v\in H^{s-\frac12}(\Omega_t)$.
\begin{prop}	\label{prop:Gron}
 Let~$s>1+\frac n2$.
 Let~$\c S_*$ be an~$H^s$ reference hypersurface, and~$\delta$ small enough, so that in the corresponding~$\Lambda_*$, the maximal angle satisfies~$s<\frac12+\frac\pi{2\overline\omega}$, and all other geometric conditions imposed above apply. Let~$A>0$.
 
 Take~$\c S_t$  a~$C^2$ family of~$H^s$ hypersurfaces  so that~$\c S_0\in\Lambda_*$, and~$v\in C^2(H^s(\Omega_t))$, satisfying~$\lA v(0)\rA_{H^{s-\frac12}(\Omega_0)}<A$. Assume that~$(\c S_t,v)$ are solutions of the equations, and that the Taylor condition~$a(t,\cdot)\geq a_0>0$ is satisfied.
 
 Then there exists a time~$T>0$, depending only on~$\Lambda_*$, $A$, $a_0$, $\underline{\omega}$, $\la\c S_0\ra_s$ and~$\lA v(0,\cdot)\rA_{H^s(\Omega_0)}$, so that for all times~$t\in[0,T]$, $\c S_t\in\Lambda_*$, $\lA v(t)\rA_{H^{s-\frac12}(\Omega_t)}<A$, and the energy~$E$ satisfies
 \[E(t)\leq E(0)+\int_0^t\c F\lp E(t')\rp\d t',\]
 where~$\c F$ is an increasing function of its argument, and depends only on~$\Lambda_*,A,a_0$, and~$\underline{\omega}$.
\end{prop}
Remark that in fact, $\Lambda_*$ and $A$ can be chosen depending only on the data. Also because the~$L^\infty$ evolution of~$a$ is controlled by the evolution of~$\c S\in\Lambda_*$ and~$v$ in~$H^{s-\frac12}$, it is easily seen that~$a_0$ can be chosen depending only on the initial data. Thus at the end the time of validity~$T$ of the Proposition only depends on the norms of the initial data. We do not write the Proposition in this way, since the point is that the function~$\c F$ in the control of the energy is uniform in a neighborhood of the the initial data in a rougher topology.

\begin{proof}
The equation being quasilinear, one needs to use control neighborhoods in rougher topologies.
\paragraph{Control neighborhoods.}
First we prove that~$\c S$ stays in~$\Lambda_*$ and~$\lA v\rA_{H^{s-\frac12}(\Omega_t)}$ stays less that~$A$ for a short time. This rests on an estimate of the Lagrangian map~$u(t,\cdot):\Omega_0\rightarrow\Omega_t$, which is the solution of
 \[\partial_tu(t,y)=v(t,u(t,y))\]
 with initial data~$u(0)=I$. It is immediate that
 \[\lA u(t,\cdot)-I\rA_{H^s(\Omega_0)}\leq C\int_0^t\lA v(t',\cdot)\rA_{H^s(\Omega_t)}\lA u(t',\cdot)\rA_{H^s(\Omega_0)}^s\d t',\]
 by taking the duality product with any test function~$f\in H^{-s}(\Omega_0)$ and writing the ODE in integral form.
 
 Take a large~$\mu$ to be chosen later, and
 \[t_0=\mathrm{sup}\lB t;\lA v(t',\cdot)\rA_{H^s(\Omega_t)}+\la\c S_{t'}\ra_s<\mu,\forall t'\in[0,t]\rB.\]
 This time is positive because of the continuity of the solution. Then 
 \[\lA u(t,\cdot)-I\rA_{H^s(\Omega_0)}\leq\mu\int_0^t\lA u(t',\cdot)\rA^s_{H^s(\Omega_0)}\d t'.\]
 Thus by ODE estimates, we find a time and a constant~$C$ depending only on~$\mu$, such that
 \[\lA u(t,\cdot)-I\rA_{H^s(\Omega_0)}\leq Ct.\]
 This implies that in local coordinates, the function~$\eta$ that parametrized~$\c S$ above~$\c S_*$ grows linearly, so that there is a time~$t_1$ depending only on~$\mu$, $\Lambda_*$ and the norms of the initial data, such that~$\c S$ stays in~$\Lambda_*$ for~$0\leq t\leq\mathrm{min}\lB t_0,t_1\rB$.
 
 The same construction, using~$u_t=v(t,u(t))$ and~$u_{tt}=(\nabla p-ge_n)\circ u$, gives that for a time~$t_2$ depending only on~$\mu$, $\Lambda_*$, $A$ and the norms of the initial data, if~$0\leq t\leq\mathrm{min}\lB t_0,t_1,t_2\rB$ then~$\lA v(t)\rA_{H^{s-\frac12}(\Omega)}<A$.

 \paragraph{Evolution of the Curl.}
 The evolution of~$\mu=Dv-(Dv)^*$ can easily be computed to be
 \begin{equation}
  \D_t\mu=-(Dv)^*\mu-\mu Dv.
 \end{equation}
 Then it is easy to obtain energy estimates,
 \begin{equation}
  \frac{\d}{\d t}\lA\mu\rA_{H^{s-1}(\Omega_t)}\leq C\lA v\rA_{H^s(\Omega_t)}\lA\mu\rA^2_{H^{s-1}(\Omega_t)}.
 \end{equation}

 \paragraph{Commutators.}
 We need to compute the commutator between~$\D_t$ and powers of~$\c N$. The one we need is
 \begin{equation}	\label{eq:comdtN}
  \lA\lb\D_t,\c N^\sigma\rb\rA_{L(H^\sigma(\c S_t);L^2(\c S_t))}\leq C\lA v\rA_{H^s(\Omega_t)},
 \end{equation}
 for~$\frac12\leq\sigma\leq s-1$, with~$C$ depending only on~$\Lambda_*$. This will be a consequence of
 \begin{equation*}
  \lA\lb\D_t,\c N\rb\rA_{L(H^r(\c S_t);H^{r-1}(\c S_t))}\leq C\lA v\rA_{H^s(\Omega_t)},
 \end{equation*}
for~$r\in[1/2,s-1/2]$, which can be proven for~$r>1$ by writing the commutator formula, and for~$r=1/2$ by weak formulation.
Then one can use the formula
\begin{equation*}
 \lb\D_t,\c N^{k+1}\rb=\lb\D_t,\c N\rb\c N^k+\c N\lb\D_t,\c N^k\rb
\end{equation*}
and interpolation to conclude.

We will also need the commutator between~$\c N$ and~$a$, 
\begin{equation}	\label{eq:comaN}
 \lA\lb a,\c N^{s-2}\rb\rA_{L(H^{s-2}(\c S_t),H^{\frac12}(\c S_t))}\leq C\lA a\rA_{H^{s-1}(\c S_t)},
\end{equation}
which can again be proven by interpolation between integer powers, those one being computed explicitly.

\paragraph{Evolution of the energy.}

Now one can tackle the evolution of the other two terms in the energy. We write
\[E_1=\int_{\c S_t}\la\c N^{s-\frac32}\D_ta\ra^2\d S\]
and
\[E_2=\int_{\c S_t}a\la\c N^{s-1}a\ra^2\d S.\]
We recall that for a function~$f$ defined on~$\c S_t$,
\begin{equation}	\label{eq:EvSur}
 \frac{\d}{\d t}\int_{\c S_t}f \d S=\int_{\c S_t}\lp\D_t f+f(\c D\cdot v^\top-\kappa v^\perp)\rp\d S,
\end{equation}

and because~$v$ is divergence-free,
\[\c D\cdot v^\top-\kappa v^\perp=\ls\nab Nv,N\rs\in H^{s-\frac32}(\c S_t)\subset L^\infty(\c S_t).\]
Thus this second term is harmless in the estimates.

First, we prove that
\begin{equation}	\label{eq:E1}
 \la\frac{\d}{\d t}E_1-\ls\c N^{s-\frac32}(\D_ta),\c N^{s-\frac32}(a\c Na)\rs_{L^2(\c S_t)}\ra\leq\c F\lp\lA v\rA_{H^s(\Omega_t)},\la\c S_t\ra_s\rp.
 \end{equation}
Since~$\D_ta\in H^{s-\frac32}(\c S_t)\subset L^\infty(\c S_t)$, we have
\[ \la\frac{\d}{\d t}E_1-\ls\D_t\c N^{s-1}(a),a\c N^{s-1}(a)\rs_{L^2(\c S_t)}\ra\leq\c F\lp\lA v\rA_{H^s(\Omega_t)},\la\c S_t\ra_s\rp,\]
and~\eqref{eq:E1} is a consequence of the commutators estimates~\eqref{eq:comdtN} and~\eqref{eq:comaN}, and of the self-adjointness of powers of~$\c N$.

Along the same lines, the commutator estimate~\eqref{eq:comdtN} proves
\begin{equation}	\label{eq:E2}
 \la\frac{\d}{\d t}E_2-\ls\c N^{s-\frac32}(\D_ta),\c N^{s-\frac32}(\D_t^2a)\rs_{L^2(\c S_t)}\ra\leq\c F\lp\lA v\rA_{H^s(\Omega_t)},\la\c S_t\ra_s\rp.
 \end{equation}
 
 Thus, using Proposition~\ref{prop:quasilin} on the equation satisfied by~$a$, we conclude 
 \[\la\frac{\d}{\d t}E\ra\leq\c F\lp\lA v\rA_{H^s(\Omega_t)},\la\c S_t\ra_s\rp.\]
 Using Proposition~\ref{prop:controlE} to control the right-hand side by a function of~$E$, we conclude the inequality of our Proposition on the interval of time~$[0,\mathrm{min}\lB t_0,t_1,t_2\rB]$. Then if we choose~$\mu$ big enough depending only on the initial data and the control neighborhoods, the control of the energy implies that~$t_0$ is bounded from below by a time~$t_*$ depending only on the initial data. Also since we have fixed~$\mu$, $t_1$ and~$t_2$ only depend on the initial data. Therefore the control is valid up to a time~$T$ as in the Proposition. 
\end{proof}

\bibliographystyle{hacm}
\bibliography{Beach}

\end{document}